\documentclass[reqno]{amsart}
\usepackage{diagrams}
\usepackage{amssymb}
\usepackage{mathrsfs}
\usepackage{cite}
\usepackage{tikz-cd}
\usepackage{MnSymbol}
\usepackage{a4wide}

\DeclareMathOperator\C{\mathbb C}
\DeclareMathOperator\Z{\mathbb Z}
\DeclareMathOperator\Q{\mathbb Q}

\newcommand{\ed}{\qed\vspace{3mm}}

\newtheorem{theorem}{Theorem}[section]
\newtheorem{lemma}[theorem]{Lemma}
\newtheorem{cor}[theorem]{Corollary}
\newtheorem{prop}[theorem]{Proposition}
\theoremstyle{definition}
\newtheorem{definition}[theorem]{Definition}
\newtheorem{example}[theorem]{Example}

\theoremstyle{remark}

\newcommand{\dontprint}[1]\relax

\newcommand{\sspan}{\operatorname{span}}

\newcommand{\coker}{\operatorname{coker}}
\newcommand{\Spf}{\operatorname{Spf}}

\newcommand{\De}{\Delta}
\newcommand{\La}{\Lambda}
\newcommand{\Ga}{\Gamma}
\newcommand{\Aut}{\operatorname{Aut}}

\newcommand{\hra}{\hookrightarrow}

\newcommand{\A}{{\mathbb A}}
\renewcommand{\AA}{{\mathcal A}}
\newcommand{\HH}{{\mathcal H}}
\newcommand{\wt}{\widetilde}
\newcommand{\ot}{\otimes}
\newcommand{\Hom}{\operatorname{Hom}}

\newcommand{\Om}{\Omega}
\newcommand{\TT}{{\mathcal T}}
\newcommand{\NN}{{\mathcal N}}

\newcommand{\CC}{{\mathcal C}}
\newcommand{\EE}{{\mathcal E}}

\newcommand{\FF}{{\mathcal F}}
\newcommand{\KK}{{\mathcal K}}

\newcommand{\GG}{{\mathcal G}}
\newcommand{\LL}{{\mathcal L}}
\newcommand{\MM}{{\mathcal M}}
\newcommand{\OO}{{\mathcal O}}
\newcommand{\PP}{{\mathcal P}}

\newcommand{\WW}{{\mathcal W}}
\newcommand{\si}{\sigma}
\newcommand{\de}{\delta}
\newcommand{\sub}{\subset}
\newcommand{\Spec}{\operatorname{Spec}}
\newcommand{\Res}{\operatorname{Res}}
\newcommand{\ov}{\overline}
\newcommand{\im}{\operatorname{im}}
\newcommand{\om}{\omega}
\newcommand{\la}{\lambda}
\renewcommand{\a}{\alpha}
\renewcommand{\b}{\beta}
\newcommand{\irr}{\operatorname{irr}}

\newcommand{\G}{{\mathbb G}}
\renewcommand{\th}{\theta}

\newcommand{\glue}{{\operatorname{glue}}}
\newcommand{\eps}{\epsilon}

\title{Extended clutching construction for the moduli of stable curves}
\author{Alexander Polishchuk}
\thanks{Partially supported by the NSF grants
DMS-2001224 and DMS-2349388, by the Simons Travel grant MPS-TSM-00002745, and within the framework of the HSE University Basic Research Program.}
\dedicatory{To David Kazhdan, with admiration}

\address{
    Department of Mathematics, 
    University of Oregon, 
    Eugene, OR 97403, USA; National Research University Higher School of Economics, Moscow, Russia}

\numberwithin{equation}{section}

\begin{document}

\begin{abstract}
We give a description of the formal neighborhoods of the components of the boundary divisor in the Deligne-Mumford moduli stack $\ov{\MM}_g$ of stable curves
in terms of the {\it extended clutching construction} that we define. This construction can be viewed as a formal version of the analytic plumbing construction. The advantage of our formal construction
is that we can control the effect of changing formal parameters at the marked points that are being glued. As an application, we prove that the infinitesimal neighborhood
of the boundary component $\De_{1,1}$ in $\ov{\MM}_2$ is canonically isomorphic to the infinitesimal neighborhood of the zero section in the normal bundle.
As another application, we show how to study the period map near the boundary components $\De_{g_1,g_2}$
in terms of the coordinates coming from our extended clutching construction. 
\end{abstract}

\maketitle

\section{Introduction}

Let $\ov{\MM}_g$ be the moduli stack of stable curves of genus $g\ge 2$. Consider the component $\De_{g_1,g_2}\sub\ov{\MM}_g$
of the boundary divisor corresponding to reducible curves with components of genus $g_1$ and $g_2$, where $g=g_1+g_2$.
We are interested in a description of the $n$th order thickening $\De_{g_1,g_2}^{(n)}$ of $\De_{g_1,g_2}$ in $\ov{\MM}_g$ for any $n\ge 1$.

Recall that there is a natural ``clutching" surjective finite morphism
$$\xi_{g_1,g_2}:\ov{\MM}_{g_1,1}\times \ov{\MM}_{g_2,1}\to \De_{g_1,g_2}$$
which for $g_1\neq g_2$ is generically an isomorphism (see \cite[ch.\ XII, Prop.\ 10.11]{ACG2}).

For every $n\ge 1$, let us set $S_n=\Spec(\Z[q]/(q^{n+1}))$.
Let $\ov{\MM}^{(\infty)}_{g,1}\to \ov{\MM}_{g,1}$ denote the moduli space of stable curves of genus $g$ with one marked point together with a formal
parameter at the marked point.

We show that the clutching morphism extends naturally to a collection of compatible morphisms
$$\xi_{g_1,g_2}^{(n)}:\ov{\MM}^{(\infty)}_{g_1,1}\times \ov{\MM}^{(\infty)}_{g_2,1}\times S_n\to \De_{g_1,g_2}^{(n)},$$
for $n\ge 0$.
The corresponding family of stable curves is defined by gluing the constant deformations of the complements
to the marked points with the standard deformation of the node $x_1x_2=q$.
This can be viewed as an algebraic analog of the {\it plumbing construction}, well known in the Teichm\"uller space approach to the moduli of curves
(see \cite{HK}).

The price one pays for getting these extended morphisms is the need to fix formal parameters to be glued.
However, we show that these choices can be controlled in a way compatible with the morphisms $\xi_{g_1,g_2}^{(n)}$.
The setup for this is the well known action of the Lie algebra of vector fields on the moduli spaces $\ov{\MM}^{(\infty)}_{g,1}$,
or more precisely of the ind-group $\GG$ whose $R$-points are continuous automorphisms of $R(\!(t)\!)$ (see \cite{Kon}, \cite[Sec.\ 17.3]{BZF}).



We construct a group scheme 
$\GG_{\glue,n}$ over $S_n$ together with an embedding 
$$\GG_{\glue,n}\hra (\GG\times \GG)_{S_n},$$
so that we can consider the induced action of $\GG_{\glue,n}$ on 
$$\ov{\MM}^{(\infty)}_{g_1,g_2,n}:=\ov{\MM}^{(\infty)}_{g_1,1}\times \ov{\MM}^{(\infty)}_{g_2,1}\times S_n.$$
Our main result (see Theorem \ref{inf-nbhd-thm} below) states that the extended clutching morphism $\xi_{g_1,g_2}^{(n)}$
factors through the quotient by $\GG_{\glue,n}$ and that the map from this quotient to the $n$th neighborhood of $\De_{g_1,g_2}$
is generically a $\G_m$-torsor for $g_1\neq g_2$ (there is an extra double covering in the case $g_1=g_2$).

The same construction works for the component $\De_0\sub \ov{\MM}_g$ of the boundary divisor corresponding to curves with a non-separating node.
Namely, the clutching morphism
$$\xi_0:\ov{\MM}_{g-1,2}\to \De_0$$
extends to a collection of morphisms
$$\xi_0^{(n)}:\ov{\MM}^{(\infty)}_{g-1,2}\times S_n\to \De_0^{(n)},$$
where we consider the moduli space with choices of parameters at both marked points.
We still have the action of $\GG_{\glue,n}$ on $\ov{\MM}^{(\infty)}_{g-1,2}$, so that $\xi_0^{(n)}$ factors through the quotient.


To formulate the statement more precisely let us introduce some notation.
Let $\ov{\MM}^{\irr,(\infty)}_{g_1,g_2,n}\sub \ov{\MM}^{(\infty)}_{g_1,g_2,n}$ be the open locus corresponding to pairs of irreducible curves, and let
$U\De_{g_1,g_2}^{(n)}\sub \De_{g_1,g_2}^{(n)}$ be the corresponding open locus in $\De_{g_1,g_2}^{(n)}$ corresponding to curves with only one separating node.
In the case $g_1=g_2=g/2$ we consider the \'etale double coverings
$$\EE_{g/2,g/2}\to U\De_{g/2,g/2}, \ \ \EE_{g/2,g/2}^{(n)}\to U\De_{g/2,g/2}^{(n)}$$
corresponding to a choice of one of the components of the stable curve with a separating node.
Similarly, we consider the open loci $\MM^{(\infty)}_{g-1,2}\sub \ov{\MM}^{(\infty)}_{g-1,2}$ of smooth curves,
 $U\De_0^{(n)}\sub \De_0^{(n)}$ of curves with only one node, and
the \'etale double coverings 
$$\EE_0\to U\De_0, \ \ \EE_0^{(n)}\to U\De_0^{(n)}$$
corresponding to ordering the preimages of the node in the normalization.
For $g_1\neq g_2$ we set $\EE_{g_1,g_2}^{(n)}=U\De_{g_1,g_2}^{(n)}$.
Note that in the terminology of \cite{ACG2} these stacks correspond to choices of a $\Ga$-structure, where $\Ga$ is the graph with two vertices marked by $g_1$ and $g_2$
and an edge between them, in the case of $\EE_{g_1,g_2}$, and $\Ga$ is the graph with the single vertex and a loop in the case of $\EE_0$.

\begin{theorem}\label{inf-nbhd-thm}
The morphism $\xi_{g_1,g_2}^{(n)}$ (resp., $\xi_0^{(n)}$)
factors through a morphism
$$[\ov{\MM}^{(\infty)}_{g_1,g_2,n}/\GG_{\glue,n}]\to \De_{g_1,g_2}^{(n)}$$
(resp., 
$[\ov{\MM}^{(\infty)}_{g-1,2}/\GG_{\glue,n}]\to \De_0^{(n)}$).
The restriction of this morphism,
$$[\ov{\MM}^{\irr,(\infty)}_{g_1,g_2,n}/\GG_{\glue,n}]\to U\De_{g_1,g_2}^{(n)}$$
(resp., $[\MM^{(\infty)}_{g-1,2}/\GG_{\glue,n}]\to U\De_0^{(n)}$)
is canonically identified with the $\G_m$-torsor over $\EE_{g_1,g_2}^{(n)}$ corresponding to the pull-back of 
$\OO(\De)$ on $\ov{\MM}_g$.
\end{theorem}

As an application, in the case $g_1=g_2=1$ we can identify the $n$th infinitesimal neighborhood of the boundary divisor with the one in the normal line
bundle.
In the case $g_1=1$, $g_2>1$, we get a presentation of the $1$st infinitesimal neighborhood as a quotient by $\G_a\rtimes \G_m$.

\begin{theorem}\label{g1-thm} Let us work over $\Spec \Q$.

\noindent
(i) 
Let us denote by $\NN_{\EE}$ the pull-back of the normal bundle to $\EE_{1,1}\to \De_{1,1}$ 
(resp., $\EE_0\to U\De_0$). 
Then there is a natural isomorphism of $\EE_{1,1}^{(n)}$ (resp., $\EE_0^{(n)}$)
with the $n$th infinitesimal neighborhood of the zero section in the line bundle $\NN_{\EE}$ over $\EE_{1,1}$ (resp., over $\EE_0$).

\noindent
(ii) For any $g\ge 3$, there is a natural identification 
$$U\De_{1,g-1}^{(1)}\simeq [(\ov{\MM}_{1,1}^{(1)}\times \ov{\MM}_{g-1,1}^{\irr,(2)}\times \Spec(\Q[q]/(q^2)))/(\G_a\rtimes \G_m^2)],$$
where $\ov{\MM}_{g-1,1}^{\irr,(2)}\sub \ov{\MM}_{g-1,1}^{(2)}$ is the locus of irreducible curves. Here
we use natural actions of $\G_a$ on $\ov{\MM}_{1,1}^{(1)}\times \Spec(\Q[q]/(q^2)))$ and on $\MM_{g-1,1}^{(2)}$
by infinitesimal translations and by the changes of the formal parameter, and the natural rescalings of parameters by $\G_m^2$.
\end{theorem}



The extended clutching morphisms provide convenient ``gluing" coordinates in the formal neighborhood of the components of the boundary divisor in $\ov{\MM}_g$. To illustrate this
we show how to compute the expansion in powers of $q$ (the gluing coordinate), defined on the formal neighborhood of $\De_{g_1,g_2}$,
of the morphism
$$\pi_*\om_{C/\MM_g}\to R^1\pi_*\C_{C/\MM_g}$$
used to define the period map (see Section \ref{periods-sec}). Here $\pi:C\to \MM_g$ is the universal curve. In particular, we get explicit formulas for the expansion modulo $q^{g_1+g_2+2}$, and in the case $g_1=g_2=1$,
a recursive procedure for computing the complete $q$-expansion.

The results of this paper admit counterparts for the moduli spaces of stable supercurves that will be treated elsewhere.
In the super case the ``gluing coordinates" near the components of the boundary divisors are useful for studying the polar behavior of the superanalog of
the Mumford's isomorphism and of the superstring supermeasure.

The paper is organized as follows. In Sec.\ \ref{gluing-group-sec} we define and study the group scheme $\GG_{\glue,n}$ which acts on deformations
of the node singularity. In Sec.\ \ref{ext-clutch-sec} we define the extended clutching construction. In Sec.\ \ref{formal-nbhd-sec} we prove Theorems \ref{inf-nbhd-thm}
and \ref{g1-thm}. Finally, in Sec.\ \ref{periods-sec} we study the period map near the boundary divisor $\De_{g_1,g_2}$ using the coordinates given by the extended clutching construction.

\medskip

\noindent
{\it Acknowledgment}. I am grateful to Giovanni Felder and David Kazhdan for discussions of the analogous picture for supercurves, which led me to consider the classical case.

\section{Changes of coordinates on a node deformation}\label{gluing-group-sec}

\subsection{The gluing group scheme}

For any commutative ring $R$ 
let us denote by $\Aut R[\![t]\!]$ the group of continuous automorphisms of $\Aut R[\![t]\!]$, identical on $R$.
Similarly we denote by $\Aut R(\!(t)\!)$ the group of continuous automorphisms of $R(\!(t)\!)$,
identical on $R$. 
The Lie algebra of this group functor is $\WW$, the Witt algebra over $\Z$. 
It has a topological basis $L_n=z^{n+1}\frac{d}{dz}$, $n\in \Z$, and the bracket is given by
$$[L_m,L_n]=(m-n)L_{m+n}$$
(so an element of $\WW$ is an infinite series $\sum_{n\ge -N}c_nL_n$).

Let us consider the induced $\Z(\!(q)\!)$-linear bracket on $\WW\ot\Z(\!(q)\!)$. 
We are interested in a completion of a certain $\Z[\![q]\!]$-lattice in $\WW\ot\Z(\!(q)\!)$.

\begin{definition} Let us consider the $\Z[\![q]\!]$-submodule $\WW^0\sub \WW\ot\Z(\!(q)\!)$ with the basis
$(M_n)_{n\in\Z}$, where
$$M_n=L_n \text{ for } n\ge 0, \ \  M_{-n}=q^nL_{-n} \text{ for } n>0.$$
We also set $\WW^0_n:=\WW^0/q^{n+1}\WW^0$. Now let $\widehat{\WW}^0$ (resp., $\widehat{\WW}^0_n$) denote the completion of $\WW^0$ (resp., $\WW^0_n$)
with respect to the filtration by $\Z[\![q]\!]$-submodules generated by $M_n$ with $|n|>N$, where $N=1,2,\ldots$.
\end{definition}

In other words, in $\widehat{\WW}^0$ we allow infinite sums in both directions, $\sum_{i\in \Z} c_iM_i$. Using the relation
$$[M_i,M_{-j}]=q^{\min(i,j)}(i+j)M_{i-j}, \text{ for } i\ge 0,j\ge 0,$$
it is easy to see that the Lie bracket on $\WW^0$ induces a well defined Lie bracket on $\widehat{\WW}^0$.
We are mostly interested in the induced Lie algebra structure on $\widehat{\WW}^0_n$, linear over $A_n:=\Z[q]/(q^{n+1})$.

Note that we have a natural involution preserving the bracket, 
\begin{equation}\label{kappa-Lie-involution}
\kappa:\widehat{\WW}^0\to \widehat{\WW}^0: \kappa(M_n)=-M_{-n} \text{ for } n\in\Z.
\end{equation}

Also, the Lie algebra $\widehat{\WW}^0_n$ has a natural equivariant structure with respect to rescaling of $q$.
In other words, if we think of $\widehat{\WW}^0_n$ as a sheaf of Lie algebras over $\A^1_q$ then
it has a natural $\G_m$-equivariant structure. Namely, for any $A_n$-algebra $R$ and any unit $\la\in R^*$,
we denote by $R(\la q)$ the same ring $R$ with the new $A_n$-algebra structure that replaces $q\in R$ by $\la q$.
Then we have an action of $\la$,
$$\widehat{\WW}^0\ot_{A_n} R\rTo{\a_\la} \widehat{\WW}^0\ot_{A_n} R(\la q)$$
such that $\a_\la(M_n)=\la^nM_n$ for $n\ge 0$ and $\a_\la(M_n)=M_n$ for $n<0$.

Given a $\Z[q]/(q^{n+1})$-algebra $R$, we are going to show that
the natural Lie action of $\widehat{\WW}^0_n$ on $R(\!(t)\!)$ (where $M_i$ with $i\le -n-1$ acts trivially) integrates to an action of a certain
group scheme.

For a $\Z[q]/(q^{n+1})$-algebra $R$ let us set
$$A(R):=R[\![t_1,t_2]\!]/(t_1t_2-q).$$

\begin{definition} We define the group
$\GG_{\glue,n}(R)$ as the group of continuous automorphisms $\a$ of $A(R)$, identical on $R$, such that 
$$\a(t_1)=t_1\cdot u, \ \ \a(t_2)=t_2\cdot u^{-1},$$
for some unit $u$ in $A(R)$.
\end{definition}

The fact that the condition defining $\GG_{\glue,n}(R)$ is closed under product is easy to check:
if 
$$\b(t_1)=t_1\cdot v, \ \ \b(t_2)=t_2\cdot v^{-1},$$
then
$$(\a\circ\b)(t_1)=t_1\cdot (u\cdot\a(v)), \ \ (\a\circ\b)(t_2)=t_2\cdot (u^{-1}\cdot\a(v)^{-1}),$$
and we observe that $u\cdot \a(v)$ is again a unit in $A(R)$.

We have a natural involution 
$$\kappa:\GG_{\glue,n}(R)\to \GG_{\glue,n}(R): \a\mapsto \si \a \si^{-1},$$
where $\si$ swaps $t_1$ and $t_2$.

In the case $n=0$, i.e., $q=0$, the group $\GG_{\glue,0}(R)$ is generated by two commuting subgroups,
$\Aut R[\![t_1]\!]$ and $\Aut R[\![t_2]\!]$ (given by the condition that $u$ depends only on $t_1$ or only on $t_2$),
which intersect by $R^*$. In fact, it is easy to see that there is an isomorphism of groups
$$\GG_{\glue,0}(R)\simeq \Aut_1 R[\![t_1]\!]\times R^*\times \Aut_1 R[\![t_2]\!],$$
where $\Aut_1 R[\![t]\!]\sub \Aut R[\![t]\!]$ denotes the subgroup of automorphisms $\a$ such that $\a(t)\equiv t\mod (t^2)$.

We have a natural homomorphism of $R$-algebras
\begin{equation}\label{iota-def}
\iota:A(R)\to R(\!(t_1)\!)\oplus R(\!(t_2)\!):t_1\mapsto (t_1,q/t_2), \ t_2\mapsto (q/t_1,t_2).
\end{equation}
Note that this is well-defined since $q^{n+1}=0$ in $R$. 

Furthermore, it is easy to see that there is a natural identification $A(R)[t_i^{-1}]\simeq R(\!(t_i)\!)$,
for $i=1,2$, so that $\iota$ can be identified with the map
$$A(R)\to A(R)[t_1^{-1}]\oplus A(R)[t_2^{-1}].$$

\begin{lemma}\label{A(R)-basic-lem}
(i) The morphism $\iota$ is injective (in fact, a direct summand as an $R$-submodule), and
we have an inclusion
$$R[\![t_1]\!]t_1^{n+1}\oplus R[\![t_2]\!]t_2^{n+1}\sub \im(\iota).$$

\noindent (ii)
If an element $u\in A(R)^*$ satisfies $t_1\cdot u=t_1$, $t_2\cdot u^{-1}=t_2$, then $u=1$.
The functor $R\mapsto \GG_{\glue,n}(R)$ is represented by a flat group scheme $\GG_{\glue,n}$ over $S_n=\Spec(\Z[q]/(q^{n+1}))$ 
(which as a scheme is isomorphic to
the product of $\G_m$ with the infinite dimensional affine space over $S_n$). One has
$$\GG_{\glue,n+1}\times_{S_{n+1}} S_n\simeq \GG_{\glue,n}.$$
\end{lemma}

\begin{proof} (i) First we observe  $\iota(x_1^i)=(x_1^i,0)$, $\iota(x_2^i)=(0,x_2^i)$ for $i\ge n+1$, which implies the required inclusion
into $\im(\iota)$. Thus, it is enough to prove that the map
$$\ov{\iota}:R[t_1,t_2]/(t_1^{n+1},t_2^{n+1},t_1t_2-q)\to P:=R(\!(t_1)\!)/R[\![t_1]\!]t_1^{n+1}\oplus R(\!(t_2)\!)/R[\![t_2]\!]t_2^{n+1},$$
induced by $\iota$ is injective. Next, let $P'$ be the $R$-submodule of $P$ with the basis 
$(1,0)$, $(x_1^{-i},x_2^{-i})_{i\le 1}$. Then the $R$-module $P/P'$ is free of finite rank, and it is enough to prove that the map
$$R[t_1,t_2]/(t_1^{n+1},t_2^{n+1},t_1t_2-q)\to P/P'$$
of free $R$-modules is an isomorphism. But this is true, since it becomes an isomorphism modulo $q$, and $q$ is nilpotent.

\noindent
(ii)
Indeed, for such an element we see that $u-1$ goes to zero under the localization with respect to $t_1$ and under
the localization with respect to $t_2$. Hence, $\iota(u-1)=0$, so $u=1$.
It follows that the elements of $\GG_{\glue,n}(R)$ are in bijection with $A(R)^*$.
Note that an element $u$ of $A(R)$ is a unit if and only if it has form
$$u_0+a_1t_1+a_2t_1^2+\ldots+b_1t_2+b_1t_2^2+\ldots,$$
where $u_0\in R^*$, $a_i\in R$, $b_i\in R$. This gives an isomorphism of $\GG_{\glue,n}$ with the product of $\G_m$ and an infinite-dimensional
affine space over $S_n$. The last assertion is clear.
\end{proof}



It is clear that any $\a\in \GG_{\glue,n}(R)$ induces well defined homomorphisms
$\a_i$ of $A(R)[t_i^{-1}]\simeq R(\!(t_i)\!)$, so that 
$$(\a_1,\a_2)\circ\iota=\iota\circ\a.$$
In this way we get an injective morphism of group ind-schemes
\begin{equation}\label{gluing-scheme-homomorphism-eq}
\GG_{\glue,n}(R)\to (\GG\times \GG)_{S_n}(R)=\Aut R(\!(t_1)\!) \times \Aut R(\!(t_2)\!): \a\mapsto (\a_1,\a_2).
\end{equation}
More explicitly, if $\a(t_1)=t_1\cdot u(t_1,t_2)$ then
$$\a_1(t_1)=t_1\cdot u(t_1,q/t_1), \ \ \a_2(t_2)=t_2\cdot u(q/t_2,t_2)^{-1}.$$
In particular, this map defines an action of $\GG_{\glue,n}(R)$ on $R(\!(t_1)\!)$ and on $R(\!(t_2)\!)$.
We can think of this geometrically as two actions of $\GG_{\glue,n}$ on the punctured formal disk.

\begin{example} For $r\in R$, we can consider the element $\a\in \GG_{\glue,1}(R)$ corresponding to $u=1+rt_1$.
Then we have
$$\a_1(t_1)=t_1+rt_1^2, \ \ \a_2(t_2)=t_2-qr$$
(where we use the vanishing of $q^2$),
so $\a_1$ acts trivially on the tangent space of the formal disk, while $\a_2$ corresponds to the infinitesimal translation.
\end{example}

For a unit $\la\in R^*$ let us denote by $R(\la q)$ the same ring $R$ but with the $\Z[q]/(q^{n+1})$-algebra structure
given by $\la\cdot q$. Then we have a natural map
$$a_\la:\GG_{\glue,n}(R)\to \GG_{\glue,n}(R(\la q))$$
sending $u(t_1,t_2)\in A(R)^*$ to $u(\la t_1,t_2)\in A(R(\la q))^*$.
This defines a $\G_m$-equivariant structure on $\GG_{\glue,n}$, so that it descends to a group scheme
over $\Spec(\Z[q]/(q^{n+1}))/\G_m$.

Let us denote by $\GG_1\sub \GG$ the subgroup given by $\GG_1(R)=\Aut_1 R[\![t]\!]$.

\begin{prop}\label{group-scheme-Lie-prop}
(i) The Lie algebra of $\GG_{\glue,n}$ can be naturally identified with
$\widehat{\WW}^0_n$, compatibly with the $\G_m$-equivariant structures and
with their actions on the punctured disks, so that the Lie algebra involution
\eqref{kappa-Lie-involution} is induced by the group involution $\kappa$.

\noindent
(ii) There are natural embedding of groups
$$j_1:\GG_1(R)=\Aut_1 R[\![t_1]\!]\to \GG_{\glue,n}(R), \ \ j_2:\GG'_1(R)=\Aut_1 R[\![t_2]\!]\to \GG_{\glue,n}(R), \ \ R^*\to \GG_{\glue,n}(R),$$
corresponding to the cases when the unit $u\in A(R)$ depends only on $t_1$, only on $t_2$, or is in $R^*$, respectively. 
The product map
\begin{equation}\label{Aut-decomposition}
\Aut_1 R[\![t_1]\!]\times R^*\times \Aut_1 R[\![t_2]\!]\to \GG_{\glue,n}(R)
\end{equation}
is an isomorphism of schemes.
Consider the subgroups $\Aut_n(R[\![t]\!],n)\sub \Aut_1(R[\![t]\!])$ consisting of automorphisms trivial modulo $t^{n+1}$.
Then the map
$$\Aut_n(R[\![t_1]\!])\times \Aut_n(R[\![t_2]\!])\to \GG_{\glue,n}(R)$$
is an embedding of a subgroup, and its composition with the embedding to $\GG(R)$ is the product of the natural embeddings of the factors.
\end{prop}

\begin{proof}
(i) We have to consider elements of $\GG_{\glue,n}(D)$, where $D=\Z[q,\eps]/(q^{n+1},\eps^2)$, reducing to the identity modulo $\eps$, which
are given by units of the form $u=1+\eps a$, where $a\in A(\Z[q]/(q^{n+1}))$. The basis element $M_i$ for $i\ge 0$ correspond to $u_i=1+\eps t_1^i$, while
$M_{-i}$ for $i\ge 0$ corresponds to $u_{-i}=1+\eps t_2^i$. A direct computation of the commutators shows that these elements satisfy the relations of $\widehat{\WW}^0_n$.

\noindent
(ii) To prove that \eqref{Aut-decomposition} is injective we need to check that the intersection
of subgroups $\Aut_1 R[\![t_1]\!]$ and $\Aut R[\![t_2]\!]$ in $\GG_{\glue,n}(R)$ is trivial. Indeed, the corresponding unit $u\in A(R)^*$ belongs 
at the same time to $1+t_1R[\![t_1]\!]$ and to $R[\![t_2]\!]$, hence $u=1$. 

To check surjectivity, 
let us use the induction on $n$. If $q=0$ then the assertion is clear. Assume that $n\ge 1$ and we know the result modulo $q^n$,
so for any element $g\in \GG_{\glue,n}(R)$, there exists a decomposition $g \mod q^n=g_1\cdot g_2$, where $g_1$ and $g_2$ belong to the $R/(q^n)$-points of the corresponding subgroups.
Lifting $g_1$ and $g_2$ to $R$-points $\wt{g}_1$ and $\wt{g}_2$ and replacing $g$ by $\wt{g}_1^{-1}g\wt{g}_2^{-1}$, we reduce to the case when $g \mod q^n=1$.
Thus, we need to show that every unit $u\in A(R)^*$ of the form $u=1+q^na$, can be
represented in the form
$$u=u_1\cdot \a_{u_1}(u_2),$$
where $u_1\in 1+t_1R[\![t_1]\!]$, $u_2\in R[\![t_2]\!]$ and $\a_{u_1}$ is an automorphism of $A(R)$ associated with $u_1$.
We can write $a=a_1+a_2$, where $a_1\in t_1R[\![t_1]\!]$ and $a_2\in R[\![t_2]\!]$. Now we can take
$u_1=1+q^na_1$, $u_2=1+q^na_2$. Note that $q^{2n}=0$, so $u_1^{-1}=1-q^na_1$ and since $q^nt_1t_2=0$, we deduce that
$\a_{u_1}$ acts trivially on $R[\![t_2]\!]$. Thus,
$$u_1\cdot \a_{u_1}(u_2)=u_1\cdot u_2=1+q^n(a_1+a_2)=u.$$

The last assertion is straightforward.
\end{proof}

\subsection{Isomorphisms between node deformations}


\begin{prop}\label{Aut-AR-prop} 
Suppose $q,q'\in R$ are two nilpotent elements. Then any continuous isomorphism of $R$-algebras,
$$\a:R[\![x_1,x_2]\!]/(x_1x_2-q')\to R[\![x_1,x_2]\!]/(x_1x_2-q)$$
sending $(x_1)$ to $(x_1)$ and $(x_2)$ to $(x_2)$, is a composition of an isomorphism
$$x_1\mapsto \la x_1, \ x_2\mapsto x_2,$$
for uniquely defined $\la\in R^*$, where $q'=\la q$, followed by
an automorphism from $\GG_{\glue,n}(R)$.
\end{prop}
 
\begin{proof}
Since $\a((x_1x_2))=(x_1x_2)$, we have the equality of ideals $(q)=(q')$, so $\a$ induces an automorphism of
$R/q[\![x_1,x_2]\!]/(x_1x_2)$.
 
By assumption, we have
$$\a(x_1)=x_1u_1(x_1,x_2), \ \ \a(x_2)=x_2u_2(x_1,x_2).$$
Reducing modulo $q$, we deduce that $u_1=\la_1\mod (x_1,x_2)$, $u_2=\la_2\mod (x_1,x_2)$, where $\la_1$ and $\la_2$ project to units in $R/(q)$.
Hence, $\la_1$ and $\la_2$ are units in $R$, and so $u_1$ and $u_2$ are units in $A(R)$.
Post-composing $\a$ with an element of $\GG_{\glue,n}(R)$ corresponding to $u_2$, we obtain an isomorphism
$$\a'(x_1)=x_1u(x_1,x_2), \ \ \a'(x_2)=x_2,$$
where $u=u_1u_2$ is a unit in $A(R)$.
Let us write
$$u=\la+x_1f_1(x_1)+x_2f_2(x_2),$$
where $\la\in R^*$.
Then the condition 
$$q'=\a'(x_1)\a'(x_2)=x_1x_2u=qu$$
implies that 
$$q'=\la q, \ \ qf_1=0, \ \ qf_2=0.$$
It follows that
$$x_1x_2f_2=0,$$
so we can replace $u$ with $u'=\la+x_1f_1(x_1)$ and still have $\a'(x_1)=x_1u'$.
Since we also have 
$$x_2x_1f_1=0,$$
if we write $u'=\la\cdot u_1$, then we get
$$x_2\cdot (u_1-1)=0.$$
Hence, 
$$x_2\cdot u_1^{-1}=x_2+x_2(u_1^{-1}-1)=x_2,$$
and the automorphism $x_1\mapsto x_1\cdot u_1$, $x_2\mapsto x_2$ belongs to $\GG_{\glue,n}(R)$.
Thus, $\a'$ is the composition of the isomorphism $x_1\mapsto \la x_1$, $x_2\mapsto x_2$ with an element of $\GG_{\glue,n}(R)$,
as claimed.

To prove the uniqueness of $\la$ we have to check that if 
$$\la\cdot x_1=x_1\cdot u, \ \ x_2=x_2\cdot u^{-1}$$
for some unit $u$ in $A(R)$ and some $\la\in R^*$, then $\la=1$.
Writing 
$$u=\la_1+x_1f_1(x_1)+x_2f_2(x_2),$$ 
where $\la_1\in R^*$, we get that
$$\la\cdot x_1=\la_1+x_1^2f_1(x_1)+qf_2(x_2),$$
which implies that $f_1=0$. Hence $u^{-1}$ has form
$$u^{-1}=\la_1^{-1}+x_2\wt{f}_2(x_2).$$
But then the condition $x_2=x_2\cdot u^{-1}$ implies that $\wt{f}_2=0$ and $\la_1=1$. Hence $u^{-1}=1$ and $\la=1$.
\end{proof}

\section{Extended clutching construction}\label{ext-clutch-sec}

\subsection{Construction}

Let $\pi:\wt{C}_0\to \Spec(R)$ be a proper flat family of curves over an affine base, equipped with a pair of distinct sections $p_1,p_2:\Spec(R)\to \wt{C}_0$, such that
$\pi$ is smooth along the images of $p_1$ and $p_2$ (abusing the notation we will denote these images also by $p_1$ and $p_2$).
Then one can define a new family $C_0\to \Spec(R)$ obtained from $\wt{C}_0$ by gluing $p_1$ with $p_2$ into a node.

We are interested in a generalization of this construction that gives a deformation of the nodal curve $C_0$. For this we fix an additional data, namely, the formal parameters
$x_1$ and $x_2$ at $p_1$ and $p_2$, respectively.


We also assume that an element $q\in R$ is fixed such that $q^{n+1}=0$. Let us denote by $C_{0,R/q}\to \Spec(R/q)$ the curve obtained from $C_0$ by the base change $R\to R/q$.
We will construct a family of curves $\CC$ over $R$, deforming $C_{0,R/q}$.
Loosely speaking, $\CC$ will be glued from 
$\wt{C}_0\setminus\{p_1,p_2\}$
and $\Spf(R[\![x_1,x_2]\!]/(x_1x_2-q))$ along $\Spf(R(\!(x_1)\!)\oplus R(\!(x_2)\!))$, i.e., 
the union of formal punctured disks in $\wt{C}_0$ around $p_1$ and $p_2$.

Let us give a more precise construction of $\CC$. 
We use the homomorphism of $R$-algebras,
$$\iota:A(R)=R[\![x_1,x_2]\!]/(x_1x_2-q)\to R(\!(x_1)\!)\oplus R(\!(x_2)\!): f(x_1,x_2)\mapsto (f(x_1,q/x_1), f(q/x_2,x_2))$$
(see \eqref{iota-def}).
Recall that by Lemma \ref{A(R)-basic-lem}, $\iota$ is injective and
\begin{equation}\label{im-iota-inclusion-eq}
R[\![x_1]\!]x_1^{n+1}\oplus R[\![x_2]\!]x_2^{n+1}\sub \im(\iota).
\end{equation}

Let $U\sub \wt{C}_0$ be an affine neighborhood of $\{p_1,p_2\}$ in $\wt{C}_0$. 
The expansion in formal parameters gives a homomorphism
$$\kappa: \OO(U\setminus \{p_1,p_2\})\to R(\!(x_1)\!)\oplus R(\!(x_2)\!).$$

\begin{lemma}\label{Dn-flat-lem}
The map
$$\OO(U\setminus \{p_1,p_2\})\to (R(\!(x_1)\!)\oplus R(\!(x_2)\!))/\im(\iota)$$
induced by $\kappa$, is surjective.
\end{lemma}

\begin{proof} 
First, we note that the image of $\iota$ is contained in $x_1^{-n}R[\![x_1]\!]\oplus x_2^{-n}R[\![x_2]\!]$. Thus, it is enough to check that for each $N\ge n$,
the map
$$H^0(U,\OO_U(Np_1+Np_2))\to x_1^{-N} R[\![x_1]\!]\oplus x_2^{-N} R[\![x_2]\!]/\im(\iota)$$
is surjective. Note that due to the inclusion \eqref{im-iota-inclusion-eq}, the target of the above map is finitely generated as $R$-module.
Thus, it is enough to check surjectivity modulo $q$. But then the required surjectivity follows 
from surjectivity of the map on global sections induced by the morphism of coherent sheaves on $U$,
$$\OO_U(Np_1+Np_2)\to \OO_U(Np_1+Np_2)/\OO_U(-p_1-p_2).$$
\end{proof}

Now let us define the $R$-subalgebra $A$ in $\OO(U\setminus\{p_1,p_2\})$ as the following fibered product:
\begin{equation}\label{A-fibered-product-diag}
\begin{diagram}
A&\rTo{}& \OO(U\setminus\{p_1,p_2\})\\
\dTo{\wt{\kappa}}&&\dTo{\kappa}\\
R[\![x_1,x_2]\!]/(x_1x_2-q)&\rTo{\iota}&R(\!(x_1)\!)\oplus R(\!(x_2)\!)
\end{diagram}
\end{equation}

\begin{lemma}\label{Dn-localization-lem} Let $f\in \OO(U)$ be such that $\kappa(f)=(x_1^{n+1}+\ldots,x_2^{n+1}+\ldots)$. Then $f\in A$ and
the localization $A[f^{-1}]$ is naturally isomorphic to $\OO(U\setminus Z(f))$, where $Z(f)$ is the divisor of zeros of $f$. Locally near $\{p_1,p_2\}$,
the divisor $Z(f)$ is supported on $\{p_1,p_2\}$.
\end{lemma}

\begin{proof}
For the first assertion, it is enough to check that the map obtained from $\iota$ by localization with respect to $f$ is an isomorphism.
Since the multiplication by $\iota(f)$ is invertible on $R(\!(x_1)\!)\oplus R(\!(x_2)\!)$, we have to show that for every
element $(p,q)\in R(\!(x_1)\!)\oplus R(\!(x_2)\!)$, one has $\iota(f^N)(p,q)\in \im(\iota)$ for some $N$. But this immediately follows from the inclusion \eqref{im-iota-inclusion-eq}.

For the last assertion, we can assume that $q=0$. It is enough to check that locally near $\{p_1,p_2\}$, one has an inclusion of ideals $\OO_U(-(n+1)p_1-(n+1)p_2)\sub (f)$.
But this follows from the fact that this inclusion holds in a formal neighborhood of $\{p_1,p_2\}$. 
\end{proof}

Lemma \ref{Dn-localization-lem} implies that one can shrink $U$ so that $Z(f)$ is supported at $\{p_1,p_2\}$ and the morphism
$$U\setminus \{p_1,p_2\}\to \Spec(A)$$ 
is an open embedding. Therefore, we can define the $R$-scheme $\CC$ by gluing two open subsets, $\wt{C}_0\setminus\{p_1,p_2\}$
and $\Spec(A)$ along $U\setminus \{p_1,p_2\}$.

\begin{prop}\label{clutch-constr-prop}
The scheme $\CC=\CC(\wt{C}_0,p_1,p_2,x_1,x_2;q)$ is flat and proper over $\Spec(R)$. It does not depend on a choice of an affine neighborhood $U$ up to a canonical isomorphism.
Given a homomorphism $\phi:R\to R'$, setting $q'=\phi(q)$, we have a natural isomorphism
$$\CC(\wt{C}_0\times_{\Spec R}\Spec R',p'_1,p'_2;q')\simeq \CC(\wt{C}_0,p_1,p_2;q)\times_{\Spec R} \Spec R',$$
where $p'_1$ and $p'_2$ are obtained from $p_1$ and $p_2$ by the base change. In particular, the base change $\CC\times_{\Spec R} \Spec R/q$ gives
the curve $C_{0,R/q}$ obtained by gluing the marked points $p_1$ and $p_2$ in $\wt{C}_{0,R/q}$ into a node $\nu\in C_{0,R/q}$.
\end{prop}

\begin{proof}
To prove flatness it is enough to check that $A$ is flat over $R$. But this follows from the exact sequence of $R$-modules
$$0\to A\to \OO(U\setminus\{p_1,p_2\})\to \coker(\iota)\to 0,$$
where surjectivity follows from Lemma \ref{Dn-flat-lem}. Since $\coker(\iota)$ is a free $R$-module, it follows that $A$ is flat over $R$.

Next, to show properness, it is enough to consider the corresponding family over $\Spec(R/q)$, i.e., we can assume that $q=0$.
In this case the image of $\iota$ is contained in $R[\![x_1]\!]\oplus R[\![x_2]\!]\sub R(\!(x_1)\!)\oplus R(\!(x_2)\!)$, so $A$ is contained in
$\OO(U)\sub \OO(U\setminus\{p_1,p_2\})$. This implies that
we have an affine morphism $f:\wt{C}_0\to \CC$, which is an isomorphism over $\wt{C}_0\setminus \{p_1,p_2\}\sub \CC$.
It is enough to check that in fact $f$ is a finite morphism. Indeed, we need to check that $\OO(U)$ is a finite $A$-module. 
From the definition of $\iota$, we see that $A$ is exactly the subring of $\phi\in \OO(U)$ such that $p_1^*\phi=p_2^*\phi\in R$.
Thus, choosing a function $\phi_0\in \OO(U)$ such that $p_1^*\phi_0=1$ and $p_2^*\phi_0=0$ (such $\phi_0$ exists since $U$ is affine),
we obtain
$$\OO(U)=A\oplus R\cdot \phi_0,$$
which implies that $\OO(U)$ is a finite $A$-module, as claimed.

Let us see what happens if we replace $U$ by a smaller neighborhood $U'\sub U$ of $\{p_1,p_2\}$. 
Let us denote the corresponding glued curve by $\CC(U')$. Note that we have a natural morphism $\CC(U')\to \CC=\CC(U)$.

We can choose a function $g$ on $U$ such that $g|_{U\setminus U'}=0$ and 
$$\kappa(g)=(1+O(x_1^{n+1}),1+O(x_2^{n+1})).$$
This implies that $g$ belongs $A$ and the localization $A_g$ fits into a fibered product similar to \eqref{A-fibered-product-diag}, but with
$U$ replaced by $U\setminus Z(g)\sub U'$, where $Z(g)$ is the zero locus of $g$. 
Hence, the morphism $\CC(U\setminus Z(g))\to \CC$ is an open embedding. 
Since it is also proper, it is an embedding of a connected component. Looking at the induced map over $R/q$, we see that it is an isomorphism.
Thus, the map $\CC(U')\to \CC$ has the right inverse $\CC\simeq \CC(U\setminus Z(g))\to \CC(U')$. Repeating the same argument
with $U$ replaced by $U'$ and $U'$ replaced by $U\setminus Z(g)$, we see that $\CC(U')\to \CC$ is an isomorphism.

Finally, to show the compatibility with the base change $R\to R'$, it is enough to check the corresponding base change of the diagram \eqref{A-fibered-product-diag}
is still cartesian. But this follows from the fact that the bottom arrow of this diagram can be replaced by a map of free $R$-modules, upon taking quotients by
$(x_1^{n+1},x_2^{n+1})$.
\end{proof}

Note that we have a natural closed embedding $\nu:\Spec(R/q)\to C_{0,R/q}\sub \CC$, which corresponds to the composed homomorphism
$$A\rTo{\wt{\kappa}} R[\![x_1,x_2]\!]/(x_1x_2-q)\to R/q$$
sending $x_1$ and $x_2$ to $0$.
The curve $\CC$ comes with a flat covering $U_1\sqcup U_2\to \CC$, where $U_1=\wt{C}_0\setminus\{p_1,p_2\}$, and
$U_2=\Spf(R[\![x_1,x_2]\!]/(x_1x_2-q))$, where we identify $U_2$ with the formal neighborhood of of the node $\nu$.

From the above construction we obtain morphisms
\begin{equation}\label{ext-clutch-mor-1}
\ov{\MM}_{g_1,1}^{(\infty)}\times \ov{\MM}^{(\infty)}_{g_2,1}\times \Spec \Z[q]/(q^{n+1})\to \ov{\MM}_{g},
\end{equation}
\begin{equation}\label{ext-clutch-mor-2}
\ov{\MM}^{(\infty)}_{g-1,2}\times\Spec \Z[q]/(q^{n+1})\to \ov{\MM}_g,
\end{equation}
for $g_1\ge 1$, $g_2\ge 1$, $g_1+g_2=g\ge 2$. Namely, for a ring $R$, an $R$-object of $\ov{\MM}_{g_1,1}^{(\infty)}\times \ov{\MM}^{(\infty)}_{g_2,1}\times \Spec \Z[q]/(q^{n+1})$
gives a pair of curves with marked points $(C_1,p_1)$, $(C_2,p_2)$ over $R$ and an element $q_R\in R$ such that $q_R^{n+1}=0$. Thus, we can apply the above
construction to the curve $\wt{C}_0=C_1\sqcup C_2$, the marked points $p_1,p_2$, and the element $q_R$. The corresponding glued curve $\CC$ over $R$ will be stable of genus $g$, so
we obtain an object of $\ov{\MM}_g(R)$. The definition of the second morphism is similar.

\subsection{Recovering the gluing data}

\begin{prop}\label{gluing-data-equivalence-prop}
The flat proper family of curves, $\CC=\CC(\wt{C}_0,p_1,p_2,x_1,x_2;q)$, constructed in Proposition \ref{clutch-constr-prop}, 
is equipped with a closed embedding $\nu:\Spec(R/(q))\hra \CC$ of $R$-schemes and an isomorphism of $R$-algebras
$$\eta:R[\![x_1,x_2]\!]/(x_1x_2-q)\rTo{\sim} \OO(\hat{\CC}_\nu),$$
compatible with homomorphisms to $R/(q)$,
where $\hat{\CC}_\nu$ is the completion of $\CC$ along $\nu(\Spec(R/q))$. 
In other words, the ideal of the image of $\nu$ corresponds under $\eta$ to the ideal $(x_1,x_2)$. 
The correspondence
$$(\wt{C}_0,p_1,p_2,x_1,x_2)\mapsto (\CC,\nu,\eta)$$
is an equivalence of categories fibered over the category of $\Z/(q^{n+1})$-algebras.
\end{prop}

\begin{proof} 
{\bf Step 1}.
For the first statement we have to check that the map $\wt{\kappa}$ in the fibered square \eqref{A-fibered-product-diag} induces an isomorphism of
the completion of $A$ at the ideal $I=\wt{\kappa}^{-1}((x_1,x_2))$. 
Let us choose an element $f_N\in \OO(U\setminus \{p_1,p_2\})$ such that
$\kappa(f_N)=(x_1^N+\ldots,x_2^N+\ldots)$ for very large $N$. Then $f_N\in A$ and for any $a\in A$ we have 
$$\kappa(a\cdot f_N)\sub \im(\iota)\cdot \kappa(f_N)\sub R[\![x_1]\!]x_1^{n+1}\oplus R[\![x_2]\!]x_2^{n+1}\sub \iota((x_1,x_2)),$$
hence, $(f_N)\sub I$.
On the other hand, we have
$$\kappa\iota((x_1,x_2)^{2N})\sub (R[\![x_1]\!]x_1^{n+1}\oplus R[\![x_2]\!]x_2^{n+1})\cdot \kappa(f_N)\sub\kappa(A\cdot f_N).$$
Therefore, $I^{2N}\sub (f_N)$, and so the ideals $(f_N)$ and $I$ define equivalent topologies on $A$.
Now if we pass to completions with respect to the $(f_N)$-adic topology in the cartesian square \eqref{A-fibered-product-diag}, the the map
$\kappa$ will become an isomorphism. Hence, $\wt{\kappa}$ will also be an isomorphism after completion, as claimed.

\noindent
{\bf Step 2}.
Now let us start with a flat proper family of curves $\CC$ over $\Spec(R)$, together with the data $(\nu,\eta)$ as in the statement.
Let $V=\Spec(A)$ be an open affine neighborhood of the image of $\nu$, and let $I\sub A$ be the ideal of $\nu(\Spec(R/q))$, so
that $\eta$ gives an isomorphism
$$R[\![x_1,x_2]\!]/(x_1x_2-q)\simeq \hat{A}_I=\varprojlim_n A/I^n.$$
In particular, for each $N>0$ we have a surjective map 
$$A\to A/I^N\simeq (R[\![x_1,x_2]\!]/(x_1x_2-q))/(x_1,x_2)^N.$$
We can pick $f\in A$ such that $f$ maps to $x_1^{n+1}+x_2^{n+1}$ under this map for $N>n+1$.
Then we have an induced map
$$A[f^{-1}]\to R[\![x_1,x_2]\!]/(x_1x_2-q)[f^{-1}]\simeq R(\!(x_1)\!)\oplus R(\!(x_2)\!)$$
(see the proof of Lemma \ref{Dn-localization-lem}).

Now let us define the algebra $B$ from the cartesian square
\begin{diagram}
B&\rTo{}& A[f^{-1}]\\
\dTo{}&&\dTo{}\\
R[\![x_1]\!]\oplus R[\![x_2]\!]&\rTo{}&R(\!(x_1)\!)\oplus R(\!(x_2)\!)
\end{diagram}
where the bottom horizontal arrow is the natural embedding.
To see that $B$ is flat over $R$, it is enough to check that the map
$$A[f^{-1}]\to (R(\!(x_1)\!)\oplus R(\!(x_2)\!))/(R[\![x_1]\!]\oplus R[\![x_2]\!])$$
is surjective. For this we can assume that $q=0$. For every $k>0$ there exists an element
$g\in A$ such that $g$ mapsto $x_1^{d(n+1)-k} \mod (x_1,x_2)^{d(n+1)}$ in $R[\![x_1,x_2]\!]/(x_1x_2,x_1^{d(n+1)},x_2^{d(n+1)})$.
It follows that $gf^{-d(n+1)}$ maps to an element in $x_1^{-k}+x_1^{-k+1}R[\![x_1]\!]+R[\![x_2]\!]$ in $R(\!(x_1)\!)\oplus R(\!(x_2)\!)$.
Similarly, there exists an element of $A[f^{-1}]$ mapping to an element in $x_2^{-k}+R[\![x_1]\!]+x_2^{-k+1}R[\![x_2]\!]$, which proves
the claimed surjectivity.

\noindent
{\bf Step 3}.
Let $J\sub B$ denote the kernel of the homomorphism $B\to R[\![x_1]\!]\oplus R[\![x_2]\!]\to R\oplus R$.
We claim that the completion of $B$ with respect to the $J$-adic topology is precisely $R[\![x_1]\!]\oplus R[\![x_2]\!]$. 
Indeed, this is proved in the same way as in Step 1, by choosing an element $f_N$ mapping to $x_1^N+x_2^N$ for a very large $N$.

Next, we observe that an element $f\in A\sub A[f^{-1}]$ belongs to $B$ and it is easy to see that $B[f^{-1}]\simeq A[f^{-1}]$. 
Hence, we can view $\Spec(A[f^{-1}])=V\setminus Z(f)$ as an open subscheme of $\Spec(B)$.

We claim that shrinking $V$ we can assume that $Z(f)$ is supported at $\im(\nu)$. Indeed, it is enough to check that $(f)$ contains $I^N$ for some $N$,
in a formal neighborhood of $\im(\nu)$. But this follows from the fact that the ideal generated by $(x_1^{n+1}+\ldots,x_2^{n+1}+\ldots)$ in $R[\![x_1,x_2]\!]/(x_1x_2-q)$ 
contains $(x_1,x_2)^N$ for sufficiently large $N$.

Now, shrinking $V$ as above, we define the curve $\wt{C}_0$ by gluing the $R$-schemes $\Spec(B)$ and $\CC\setminus \im(\eta)$ along the $V\setminus \im(\eta)=\Spec(B[f^{-1}])$.
By construction we have a pair of disjoint sections $p_1,p_2:\Spec(R)\to \Spec(B)$, such that the completion of $B$ along each gives an algebra isomorphic to $R[\![x_i]\!]$.
As in the proof of Prop.\ \ref{clutch-constr-prop}, one can show that $\wt{C}_0$ is proper over $\Spec(R)$ and that this construction does not depend on a choice of the neighborhood $V$.

It is easy to see that the two constructions are inverse of each other, so this establishes the claimed equivalence.
\end{proof}


\subsection{Computing the determinant line bundle}

Recall that the boundary line bundle $\OO(\De)$ on $\ov{\MM}_g$ is defined as the determinant line bundle
$$\LL_{C/S}:=\det \bigl(R\pi_*(\Om_{C/S}\to \om_{C/S})[1]\bigr)$$ 
associated with any family of stable curves $\pi:C\to S$.
This line bundle is equipped with a canonical section $\th$ (see \cite[Sec.\ XIII.4]{ACG2} which is a special case of a general construction of \cite{KM}).
We will compute the pair $(\LL_{C/S},\th)$ for the family corresponding to the extended clutching construction.

\begin{prop}\label{det-line-bundle-prop}
For the family $\CC=\CC(\wt{C}_0,p_1,p_2,x_1,x_2;q)$ over $\Spec(R)$, there is a natural trivialization 
$$\tau:R\rTo{\sim}\LL_{\CC/\Spec(R)}$$ 
such that $\th=\tau(q)$. 
\end{prop}

\begin{proof}
We can compute the object $R\pi_*[\Om_{\CC/S}\to\om_{\CC/S}]$ using Cech resolutions
$$\Om_{\CC/S}\to \Om_{U_1/S}\oplus \Om_{U_2/S}\to \Om_{U_{12}/S},$$
$$\om_{\CC/S}\to \om_{U_1/S}\oplus \om_{U_2/S}\to \om_{U_{12}/S}.$$
Since the maps $\Om_{U_1/S}\to \om_{U_1/S}$ and $\Om_{U_{12}/S}\to \om_{U_{12}/S}$ are isomorphisms, we get
a natural quasi-isomorphism
$$R\pi_*(\Om_{C/S}\to \om_{C/S})\to [\Om_{U_2/S}\to \om_{U_2/S}].$$
Hence, 
$$\LL_{\CC/S}\simeq \det\bigl([\Om_{U_2/S}\to \om_{U_2/S}][1]\bigr)$$
and $\th$ is the determinant of the map $A\to B$, where $[A\to B]$ is a complex of projective finitely generated $R$-modules
quasi-isomorphic to $[\Om_{U_2/S}\to \om_{U_2/S}]$.

Now using the coordinates on $U_2=\Spf(\AA)$, where $\AA=R[\![x_1,x_2]\!]/(x_1x_2-q)$, we get
$$\Om_{U_2/S}=\AA\cdot dx_1\oplus \AA\cdot dx_2/(x_1dx_2+x_2dx_1),$$ 
$$\om_{U_2/S}=\AA\cdot e,$$
where $e=-\frac{dx_1}{x_1}=\frac{dx_2}{x_2}$.
It is easy to see that we have decompositions of $R$-modules
$$\Om_{U_2/S}=R\cdot x_1dx_2 \oplus R[\![x_1]\!]\cdot dx_1\oplus R[\![x_2]\!]\cdot dx_2,$$
$$\om_{U_2/S}=R\cdot e\oplus x_1 R[\![x_1]\!]\cdot e\oplus x_2 R[\![x_2]\!]\cdot e,$$
and the complex $[\Om_{U_2/S}\to \om_{U_2/S}]$ splits into a direct sum of
$$[R\cdot x_1dx_2\to R\cdot e]\simeq [R\rTo{q} R]$$
and of an acyclic complex. Hence, the element
$$e\ot (x_1dx_2)^{-1}\in \det[R\cdot x_1dx_2\to R\cdot e]$$
induces the required trivialization $\tau$ of $\LL_{\CC/S}$ such that $\tau(q)=\th$.
\end{proof}

\section{The formal neighborhood of the boundary divisor and the gluing group scheme}\label{formal-nbhd-sec}

\subsection{Changes of formal parameters on curves}\label{formal-parameters-sec}

Let $(C,p)$ be a family of curves over a commutative ring $R$, with a marked point $p\in C(R)$, such that
$C$ is smooth over $R$ near $p$. 

Let us denote by $\hat{\OO}_{C,p}$ the ring of functions on the completion of $C$ near $p$:
$$\hat{\OO}_{C,p}:=\varprojlim_n H^0(C,\OO_C/\OO_C(-np)).$$
Recall that a relative formal parameter $t$ at $p$ is an element of ideal of $p$ in $\hat{\OO}_{C,p}$ inducing
an isomorphism 
$$R[\![t]\!]\rTo{\sim} \hat{\OO}_{C,p}.$$
Then we have the induced isomorphism
$$R(\!(t)\!)\rTo{\sim} \KK_{C,p},$$
where $\KK_{C,p}:=\varinjlim_m \varprojlim_n H^0(C,\OO_C(mp)/\OO_C(-np)$.

Given $(C,p,t)$, where $t$ is a formal parameter at $p$, and a continuous automorphism $\a$ of $R(\!(t)\!)$, we can define new data
$(C',p',t')$ as follows. Let $U$ be an open affine neighborhood of $p$. Then $C$ is glued from $U$ and $C\setminus p$ along $U\setminus p$.
We have natural homomorphisms
$$\OO(U)\to \hat{\OO}_{C,p}, \ \ \OO(U\setminus p)\to \KK_{C,p},$$
so that we have a cartesian diagram
\begin{diagram}
\OO(U) &\rTo{} & \OO(U\setminus p)\\
\dTo{}&&\dTo{}\\
\hat{\OO}_{C,p}&\rTo{}&\KK_{C,p}
\end{diagram}
Using the formal parameter $t$ we can replace the bottom row in this diagam by
the map $R[\![t]\!]\to \KK_{C,p}$.

Now we define $C'$ by gluing the new affine curve $U'$ with $C\setminus p$ along $U\setminus p$, where $\OO(U')$ is defined as the fibered product 
\begin{equation}\label{O-U'-fib-prod}
\begin{diagram}
\OO(U') &\rTo{} & \OO(U\setminus p)\\
\dTo{}&&\dTo{}\\
R[\![t]\!]&\rTo{t\mapsto \a(t)}&\KK_{C,p}
\end{diagram}
\end{equation}
The composed homomorphism $\OO(U')\to R[\![t]\!]\to R$ of $R$-algebras defines a marked point $p'$ on $U'\sub C'$, such that
$U'\setminus p'\simeq U\setminus p$ (and $C'\setminus p'\simeq C\setminus p$).

It is easy to see that the construction of $C'$ does not depend on a choice of $U$, and
defines an action of the ind-group scheme $\GG$
on the moduli stack $\ov{\MM}^{(\infty)}_{g,1}$ of curves $C$ with a smooth marked point $p$ and a formal parameter $t$ at $p$ (see also \cite[Sec.\ 17.3]{BZF}).
Similarly, $\GG\times \GG$ acts on $\ov{\MM}^{(\infty)}_{g,2}$ by regluing at two marked marked points.

Thus, using homomorphism \eqref{gluing-scheme-homomorphism-eq} from $\GG_{\glue,n}$ to $(\GG\times\GG)_{S_n}$, we get an action of
$\GG_{\glue,n}$ on 
$$\ov{\MM}^{(\infty)}_{g_1,g_2;n}=\ov{\MM}_{g_1,1}^{(\infty)}\times\ov{\MM}_{g_2,1}^{(\infty)}\times S_n.$$

\subsection{Canonical formal parameters in genus $1$}

In this section we work over $\Q$, i.e., all rings will be $\Q$-algebras, and we set $S_n=\Spec(\Q[q]/(q^{n+1}))$.
For genus $1$ we have the following notion of {\it canonical formal parameter}.

\begin{definition}
Let $(C,p)$ be a stable curve of genus $1$ with one marked point over $S=\Spec(R)$. We say that a formal parameter $t$ at $p$ is {\it canonical}
if the formal differential $dt$ extends to a global section of the dualizing sheaf $\om_{C/S}$.
\end{definition}

Since for $(C,p)$ stable of genus $1$ the map $H^0(C,\om_C)\to \om_C|_p$ is an isomorphism, for every formal parameter $t$ at $p$ there exists a unique canonical formal parameter $t_c$ at $p$
such that $t_c\equiv t\mod (t^2)$.

Let us consider the subgroup $\GG'_1\sub \GG_{\glue,1}$ acting on $\MM^{(\infty)}_{1,h;1}$ by 
the change the formal parameters
\begin{equation}\label{genus-1-h-2nd-subgroup-eq}
(x_1,x_2)\mapsto (x_1-qc_2,x_2+c_2x_2^2+\ldots)
\end{equation}
and changing the marked point $p_1$ accordingly (see Prop.\ \ref{group-scheme-Lie-prop}(ii)).
We claim that if we assume that $x_1$ is canonical and $q^2=0$,
then this action is equivalent to the standard action of $\GG'_1$ via the second factor.

\begin{lemma}\label{g1-can-par-lem}
Let $(C,p)$ be a stable pointed curve of genus $1$ over $R$, $t$ a canonical formal parameter at $p$, and $r\in R$ an element such that $r^2=0$. 
Let us consider a reglued data $(C,p',t')$, where $t'=t+r$. Then there is a natural isomorphism $(C,p,t)\simeq (C,p',t')$.
\end{lemma}

\begin{proof}
First, we observe that since $\om_{C/S}$ is isomorphic to the pull-back of a line bundle on $S$, the natural map of $R$-modules 
$$H^0(C,\om^{-1}_{C/S})\to \om^{-1}_{C/S}|_p$$
is an isomorphism. Hence, there exists an element $\wt{v}\in H^0(C,\om^{-1}_{C/S})$ whose pairing with $dt\in H^0(C,\om_{C/S})$ is equal to $1$. 
Let $v\in H^0(C,\TT_{C/S})$ denote the global derivation of $\OO_C$, which is the image of $\wt{v}$ under the map 
$\om^{-1}_{C/S}\to \TT_{C/S}$, dual to the natural map $\Om_{C/S}\to \om_{C/S}$ (which is an isomorphism away from the nodes). 
Then we still have $v(t)=1$.

Next, we observe that since $r^2=0$, we have an automorphism $\exp(rv):f\mapsto f+rv(f)$ of $C$, acting trivially on the underlying topological space.  
Thus, if $U$ is an open affine neighborhood of $p$, then we have a commutative square
\begin{diagram}
\OO(U\setminus p)&\rTo{\exp(rv)}&\OO(U\setminus p)\\
\dTo{}&&\dTo{}\\
\KK_{C,p}&\rTo{t\mapsto t+r}&\KK_{C,p}
\end{diagram}

Thus, if $\OO(U')$ is defined by the cartesian square \eqref{O-U'-fib-prod}, with $\a(t)=t+r$, then $\exp(rv)$ induces an isomorphism between $\OO(U)$
and $\OO(U')$, compatible with the marked points $p$ and $p'$, and fitting into a commutative square
\begin{diagram}
\OO(U)&\rTo{}& \OO(U\setminus p)\\
\dTo{}&&\dTo{\exp(rv)}\\ 
\OO(U')&\rTo{}& \OO(U\setminus p)
\end{diagram}
Together with the automorphism $\exp(rv)$ of $C\setminus p$, this gives an isomorphism $C\rTo{\sim} C'$ sending $p$ to $p'$ and $t$ to $t'=t+r$.
\end{proof}

The proof of Theorem \ref{g1-thm} will be based on Propositions \ref{genus1-h-prop} and \ref{genus1-1-prop} below.

\begin{prop}\label{genus1-h-prop} 
For any $h\ge 1$, there is a natural isomorphism
$$[\ov{\MM}^{(\infty)}_{1,h;1}/\GG_{\glue,1}]\simeq [(\ov{\MM}^{(1)}_{1,1}\times \ov{\MM}_{h,1}^{(2)}\times S_1)/(\G_a\rtimes\G_m)],$$
where $\G_m$ changes $(x_1,x_2)$ to $(\la x_1,\la^{-1}x_2)$, and $\G_a$ acts by the changes of formal parameters
$(x_1,x_2)\mapsto (x_1-qr,x_2+rx_2^2)$.
\end{prop}

\begin{lemma}\label{quotient-lem}
Let a group scheme $G$ act on a scheme $X$. Assume that $G$ has two subgroups $H_1,H_2\sub G$, such that the product map
$$H_1\times H_2\rTo{\sim} G$$ 
is an isomorphism of schemes. Assume also that $Y\sub X$ is an $H_2$-invariant subscheme, such that the $H_1$-action
induces an isomorphism of schemes 
$$H_1\times Y\rTo{\sim} X,$$
i.e., $Y$ is a section for the (free) action of $H_1$ on $X$.
Then we have a natural isomorphism of quotient stacks
$$[Y/H_2] {\rTo{\sim}} [X/G].$$
\end{lemma}

\begin{proof} It suffices to prove that the map induced by the $G$-action,
$$G\times_{H_2} Y\to X,$$
is an isomorphism. We claim that the composition
$$X\to H_1\times Y\to G\times Y\to G\times_{H_2} Y$$
is the inverse. Indeed, starting from $(g,y)\in G\times Y$, consider the unique decomposition of $gy$,
$$gy=h_1\cdot y',$$
with $h_1\in H_1$ and $y'\in Y$. On the other hand, we have a unique decomposition
$$g=h'_1\cdot h_2,$$
where $h'_1\in H_1$, $h_2\in H$. Now, since $h_2y\in Y$, the equality
$$h'_1\cdot (h_2y)=h_1\cdot y'$$
implies that $h'_1=h_1$ and $y'=h_2y$. Hence, $(g,y)=(h_1h_2,h_2^{-1}y')$ projects to the same point as $(h_1,y')$ in $G\times_{H_2} Y$.
\end{proof}

\begin{proof}[Proof of Proposition \ref{genus1-h-prop}]
Let us apply Lemma \ref{quotient-lem} to $X=\ov{\MM}^{(\infty)}_{1,h;1}$, $G=\GG_{\glue,1}$, $Y\sub X$ the subscheme corresponding to the data with the parameter $x_1$
canonical, $H_2=\GG'_1\rtimes \G_m$, where $\GG'_1$ acts by \eqref{genus-1-h-2nd-subgroup-eq}, and $H_1=\GG_1$ the subgroup corresponding to reversing the roles of $x_1$ and $x_2$. Note that $\GG_1$ acts freely on $\ov{\MM}^{(\infty)}_{1,1}$ and the subscheme with $x_1$ canonical is a section of this action. This implies that
$H_1$ acts freely on $X$ and $Y$ is a section for this action. On the other hand, by Lemma \ref{g1-can-par-lem}, the action of $H_2$ preserves $Y$,
and we get an isomorphism
$$[\ov{\MM}^{(\infty)}_{1,h;1}/\GG_{\glue,1}]\simeq [(\ov{\MM}^{(1)}_{1,1}\times \ov{\MM}_{g_2,1}^{(\infty)}\times S_1)/(\GG'_1\rtimes \G_m)],$$
where $\GG'_1$ acts by transformations \eqref{genus-1-h-2nd-subgroup-eq}.
Now we notice that the subgroup of $\GG'_1$ of transformations that preserve $x_2$ modulo $(x_2^3)$ acts only on $\ov{\MM}_{h,1}^{(\infty)}$
and the quotient by this action is exactly $\ov{\MM}_{h,1}^{(2)}$. The quotient of $\GG'_1$ by this subgroup is $\G_a$, and the assertion follows.
\end{proof}

\begin{prop}\label{genus1-1-prop} 
For any $n\ge 0$, there are natural isomorphisms
$$[\ov{\MM}^{(\infty)}_{1,1;n}/\GG_{\glue,n}]\simeq [(\ov{\MM}^{(1)}_{1,1}\times \ov{\MM}_{1,1}^{(1)}\times S_n)/\G_m],$$
$$[(\ov{\MM}^{(\infty)}_{1,2}\times S_n)/\GG_{\glue,n}]\simeq [(\ov{\MM}^{(1)}_{1,2}\times S_n)/\G_m],$$
where $\G_m$ changes $(t_1,t_2)$ to $(\la t_1,\la^{-1}t_2)$.
\end{prop}

\begin{proof}
We will give the proof of the first isomorphism. The proof of the second isomorphism is similar.

We identify $Y_n:=\ov{\MM}^{(1)}_{1,1}\times \ov{\MM}_{1,1}^{(1)}\times S_n$ with a closed subscheme in
$X_n:=\ov{\MM}^{(\infty)}_{1,1;n}$ defined by the condition that both parameters $t_1$ and $t_2$ are canonical.
Note that the embedding $Y_n\sub X_n$ is a section of the natural projection $X_n\to Y_n$. It suffices to prove that the morphism
given by the action,
$$\GG_{\glue,n}\times_{\G_m} Y_n\to X_n,$$
is an isomorphism. 

We will prove this by induction on $n$. The case $n=0$ is clear, since in this case the subgroup $\GG_1\times \GG_1\sub\GG_{\glue,0}$ acts separately on two factors
$\ov{\MM}^{(\infty)}_{1,1}$. Now assume $n>0$ and the assertion holds for $n-1$. We use the fact that $\GG_{\glue,n}$ is a smooth group scheme over $S_n$, and that
$(X_{n-1},Y_{n-1},\GG_{\glue,n-1})$ are obtained by the base change $S_{n-1}\to S_n$ from $(X_n,Y_n,\GG_{\glue,n})$ in a way compatible with all the structures.

Given an object $t_\bullet=(C_1,p_1,t_1;C_2,p_2,t_2;q)$ of $X_n(R)$ we can consider the reduction modulo $q^n$, 
$$\ov{t}_\bullet\in X_n(R/(q^n))=X_{n-1}(R/(q^n)).$$
By assumption, there exists an element $\ov{g}\in \GG_{\glue,n-1}(R/(q^n))=\GG_{\glue,n}(R/(q^n))$ such that $\ov{g}\cdot \ov{t}_\bullet$ is in $Y_{n-1}(R/(q^n))$.
Let us lift $\ov{g}$ to an element $g\in \GG_{\glue,n}(R)$, and set $t'_\bullet=g\cdot t_\bullet\in X_n(R)$. Then the reductions of $t'_1$ and $t'_2$ modulo $q^n$
are canonical parameters. Let $t_{1,c}$ and $t_{2,c}$ be the canonical parameters for $(C'_1,p'_1)$ and $(C'_2,p'_2)$ such that $t_{i,c}$ agrees with $t'_i$ modulo
the square of the maximal ideal of $p'_i$. Then we obtain that
$$t_{1,c}=t'_1+q^nr_1(t'_1)^2+\ldots, \ \ t_{2,c}=t'_2+q^nr_2(t'_2)^2+\ldots,$$ 
where all the higher coefficients are divisible by $q^n$. But this gives an element 
$$g'\in \Aut_1(R(\!(t_1)\!),n)\times \Aut_1(R(\!(t_2)\!),n)\sub \GG_{\glue,n}(R)$$
such that $g'(t'_1,t'_2)=(t_{1,c},t_{2,c})$ (here we use the last part of Proposition \ref{group-scheme-Lie-prop}(ii)).

On the other hand, given an $R$-point $g\in \GG_{\glue,n}(R)$ and an object $t_\bullet=(C_1,p_1,t_1;C_2,p_2,t_2;q)$ of $Y_n(R)$, such that
$g\cdot t_\bullet$ is in $Y_n(R)$, we need to check that $g$ is in $\G_m(R)\sub \GG_{\glue,n}(R)$. By the induction assumption, this is true modulo $q^n$,
so changing $g$ by an element of $\G_m(R)$, we can assume that $g\equiv 1 \mod q^n$.
We can write $g$ in the form $g=g_1\cdot g_2\cdot \la$, with $g_i\in j_i(\Aut_1 R(\!(t_i)\!))$, $\la\in R^*$. Then we
have $g_i(t_i)\equiv t_i \mod q^n$ and $\la\equiv 1 \mod q^n$. Note that $\la\cdot t_\bullet$ is still in $Y_n(R)$, while the action of $g_1\cdot g_2$ changes
the parameters $(t_1,t_2)$ to $(g_1(t_1),g_2(t_2))$. Since these should be canonical, we deduce that $g_1=g_2=1$, as claimed. 
\end{proof}



\subsection{Connection with a $\G_m$-torsor over $\EE_{g_1,g_2}^{(n)}$}

Let $R$ be a commutative $\Z[q]/(q^{n+1})$-algebra, $(\wt{C}_0,p_1,p_2)$ a family of curves over $R$ with two distinct smooth marked points and with 
relative formal parameters $x_i$ at $p_i$. 
The extended clutching construction from Sec.\ \ref{ext-clutch-sec} associates with these data a curve $\CC=\CC(\wt{C}_0,p_1,p_2,x_1,x_2;q)$ over $R$, which
is equipped with a closed embedding $\nu:\Spec(R/q)\hra \CC$ and an isomorphism of $R$-algebras
$$\eta:R[\![x_1,x_2]\!]/(x_1x_2-q)\rTo{\sim} \OO(\hat{\CC}_\nu).$$

Now we observe that the group $\GG_{\glue,n}(R)$ acts on the gluing data through the homomorphism
$\GG_{\glue,n}(R)\to \Aut R(\!(t_1)\!)\times\Aut R(\!(t_1)\!)$ and the action of the latter group on the data $(\wt{C}_0,p_1,p_2,x_1,x_2)$
(see Sec.\ \ref{formal-parameters-sec}).
On the other hand, the group $\GG_{\glue,n}(R)$ clearly acts on the set of possible isomorphisms $\eta$ (for a fixed curve $\CC$ and fixed $\nu$).
We claim that these two actions are compatible.

\begin{lemma}\label{glue-action-lem}
For any $\Z[q]/(q^{n+1})$-algebra $R$, the equivalence between the data $(\wt{C}_0,p_1,p_2,x_1,x_2)$ and $(\CC,\nu,\eta)$ 
(see Prop.\ \ref{gluing-data-equivalence-prop})
is compatible with the $\GG_{\glue,n}(R)$-action. Furthermore,  
the trivialization $\tau=\tau_{x_1,x_2,q}:R\to \LL_{\CC/\Spec(R)}$ defined in Proposition \ref{det-line-bundle-prop},
is preserved by the $\GG_{\glue,n}(R)$-action and satisfies
$$\tau_{\la x_1,x_2,\la q}=\tau_{x_1,\la x_2,\la q}=\la^{-1}\cdot \tau_{x_1,x_2,q}$$
for $\la\in R^*$.
\end{lemma}

\begin{proof}
Recall that the curve $\CC$ is obtained by gluing $\wt{C}_0\setminus\{p_1,p_2\}$ with $\Spec(A)$, where the $R$-algebra $A$ is defined as the fibered product \eqref{A-fibered-product-diag},
and the isomorphism $\eta$ is induced by the homomorphism $\wt{\kappa}:A\to R[\![x_1,x_2]\!]/(x_1x_2-q)$. We can combine
the cartesian square \eqref{A-fibered-product-diag} defining $A$ with a commutative square 
\begin{diagram}
R[\![x_1,x_2]\!]/(x_1x_2-q)&\rTo{}& R(\!(x_1)\!)\oplus R(\!(x_2)\!)\\
\dTo{\a}&&\dTo{(\a_1,\a_2)}\\
R[\![x_1,x_2]\!]/(x_1x_2-q)&\rTo{}& R(\!(x_1)\!)\oplus R(\!(x_2)\!)
\end{diagram}
for any continuous automorphism $\a$ of $R[\![x_1,x_2]\!]/(x_1x_2-q)$ preserving the ideals $(x_1)$, $(x_2)$ and inducing the automorphisms $\a_i$ of $R(\!(x_i)\!)$.
In the case when $\a\in \GG_{\glue,n}(R)$, this immediately shows the claimed compatibility of the equivalence of Prop.\ \ref{gluing-data-equivalence-prop} with
the $\GG_{\glue,n}$-action.

Next, we want to check that $\GG_{\glue,n}$-action preserves the trivialization of the determinant bundle given in Proposition \ref{det-line-bundle-prop}.
For an automorphism $\a\in \GG_{\glue,n}$ such that $\a(x_1)=u^{-1}\cdot x_1$ and $\a(x_2)=u\cdot x_2$, for some unit $u$ in  $R[\![x_1,x_2]\!]/(x_1x_2-q)$,
one can easily check that
$$\a(\frac{dx_2}{x_2})=U\cdot \frac{dx_2}{x_2}, \ \ \a(x_1dx_2)=U\cdot x_1dx_2, \ \text{ where}$$
$$U=1+u^{-1}\frac{\partial u}{\partial x_2} x_2-u^{-1} \frac{\partial u}{\partial x_1} x_1.$$
Hence, $\a$ preserves the trivialization of determinant of the subcomplex $[R\cdot x_1dx_2\to R\cdot e]$ of $[\Om_{U_2/S}\to \om_{U_2/S}]$,
and hence it preserves the trivialization of $\LL_{\CC/S}$.

On the other hand, the rescaling $(x_1,x_2,q)\mapsto (\la\cdot x_1,x_2,\la\cdot q)$ does not change the generator $e=dx_2/x_2$ of $\om_{U_2/S}$,
and rescales $x_1dx_2$ to $\la\cdot x_1dx_2$. Hence, the corresponding trivialization of the determinant line bundle gets rescaled by $\la^{-1}$.
\end{proof}

\bigskip

\noindent
{\it Proof of Theorem \ref{inf-nbhd-thm}}.
Lemma \ref{glue-action-lem} implies that morphisms \eqref{ext-clutch-mor-1}, \eqref{ext-clutch-mor-2} factor through morphisms
\begin{equation}\label{ext-clutching-qu-mor}
[\ov{\MM}^{(\infty)}_{g_1,g_2,n}/\GG_{\glue,n}]\to \ov{\MM}_g,
\end{equation}
\begin{equation}\label{ext-clutchin-qu-mor-bis}
[(\ov{\MM}^{(\infty)}_{g-1,2}\times S_n)/\GG_{\glue,n}]\to \ov{\MM}_g
\end{equation}
We claim that the morphism \eqref{ext-clutching-qu-mor} (resp., \eqref{ext-clutchin-qu-mor-bis})
factors through $\De^{(n)}_{g_1,g_2}\sub \ov{\MM}_g$
(resp., $\De^{(n)}_0\sub \ov{\MM}_g$). Indeed, this immediately follows from Proposition \ref{det-line-bundle-prop}. Let us consider for example the boundary divisor $\De_{g_1,g_2}$.
It is defined as the vanishing locus of the section
$\th$ of the determinant line bundle, and the pull-back of this section is $q$. Since $q^{n+1}=0$, we see that \eqref{ext-clutching-qu-mor} factors through $\De^{(n)}_{g_1,g_2}\sub \ov{\MM}_g$.
Hence, the defining equation of $\De^{(n)}_{g_1,g_2}$ vanishes on the open dense locus in $\ov{\MM}^{(\infty)}_{g_1,g_2,n}$ corresponding to pairs of smooth curves, and so it
vanishes everywhere.

Let us consider the case $g_1\neq g_2$ first. Let $\FF$ denote
the moduli functor associating to a commutative ring $R$ an element $q\in R$ such that $q^{n+1}=0$ and the data $(\CC,\nu,\eta)$ as in Proposition \ref{gluing-data-equivalence-prop},
such that $\CC$ is a family of stable curves with the unique separating node given by $\nu$ and under $\eta$ the variable $x_1$ corresponds to the branch of genus $g_1$.
The equivalence of Proposition \ref{gluing-data-equivalence-prop} induces an equivalence of $\FF$ with $\ov{\MM}^{\irr,(\infty)}_{g_1,g_2,n}$.
Note that $\FF(R)$ is equipped with a natural action of $\GG_{\glue,n}(R)$ by changing the isomorphism $\eta$. In
addition, we have an action of $\G_m$ on $\FF$ that changes $(x_1,x_2,q)$ to $(\la x_1,x_2,\la q)$. 
By Proposition \ref{Aut-AR-prop}, the quotient $\ov{\MM}^{\irr,(\infty)}_{g_1,g_2,n}/\GG_{\glue,n}$
is a $\G_m$-torsor over $U\De^{(n)}_{g_1,g_2}$.
 Furthermore, by Lemma \ref{glue-action-lem}, we have a morphism of $\G_m$-torsors from this quotient to the $\G_m$-torsor of trivializations of
$\LL_{\CC/\Spec(R)}$, which is necessarily an isomorphism. This finishes the proof in the case $g_1\neq g_2$.

Similarly, from Propositions \ref{gluing-data-equivalence-prop} and  \ref{Aut-AR-prop} we get that $[\ov{\MM}^{\inf,(\infty)}_{g/2,g/2,n}/\GG_{\glue,n}]$
(resp., $[(\MM^{(\infty)}_{g-1,2}\times S_n)/\GG_{\glue,n}]$)
is a $\G_m$-torsor over $\EE^{(n)}_{g/2,g/2}$ (resp., $\EE^{(n)}_0$), and we finish by using Lemma \ref{glue-action-lem} as above.
\ed

\bigskip

\noindent
{\it Proof of Theorem \ref{g1-thm}}.
Let $X_{1,g-1}^{(n)}$ denote the $\G_m$-torsor over $U\De_{1,g-1}^{(n)}$ corresponding to the line bundle $\OO(\De)$.
Combining Theorem \ref{inf-nbhd-thm} with Proposition \ref{genus1-h-prop},
we get an isomorphism for $g>2$,
$$X_{1,g-1}^{(1)}\simeq [\ov{\MM}^{\irr,(\infty)}_{1,g-1;1}/\GG_{\glue,1}]\simeq [(\ov{\MM}^{(1)}_{1,1}\times \ov{\MM}_{g-1,1}^{\irr,(2)}\times S_1)/(\G_a\rtimes\G_m)].$$
Furthermore, as in the proof of Theorem \ref{inf-nbhd-thm}, this isomorphism is compatible with the $\G_m$-action, where $\G_m$ acts on 
$\ov{\MM}^{(1)}_{1,1}\times \ov{\MM}_{g-1,1}^{(2)}\times S_1$ by rescaling $(x_1,x_2,q)\mapsto (\la\cdot x_1,x_2,\la\cdot q)$.
Let us denote the latter action as $l(\G_m)$,
to distinguish it from the subgroup $\G_m\sub \GG_{\glue,n}$ acting by $(x_1,x_2,q)\mapsto (\la\cdot x_1,\la^{-1}\cdot x_2)$.
It follows that
\begin{equation}\label{torsor-g-isom-eq}
[X_{1,g-1}^{(1)}/\G_m]\simeq
[[\ov{\MM}^{(1)}_{1,1}\times \ov{\MM}_{1,g-1}^{\irr,(2)}\times S_1)/(\G_a\rtimes\G_m]/l(\G_m)]\simeq [(\ov{\MM}^{(1)}_{1,1}\times \ov{\MM}_{1,g-1}^{\irr,(2)}\times S_1)/(\G_a\rtimes l(\G_m)\times r(\G_m))],
\end{equation}
where $r(\G_m)$ acts by  $(x_1,x_2,q)\mapsto (x_1,\la\cdot x_2,\la\cdot q)$.

Similarly, in the case $g=2$, we have $\G_m$-torsors $X_{1,1}^{(n)}\to \EE_{1,1}^{(n)}$ and
$X_0^{(n)}\to \EE_0^{(n)}$, associated with $\OO(\De)$, and we have isomorphisms
$$X_{1,1}^{(n)}\simeq [\ov{\MM}^{(\infty)}_{1,1;1}/\GG_{\glue,n}]\simeq [(\ov{\MM}^{(1)}_{1,1}\times \ov{\MM}_{1,1}^{(1)}\times S_n)/\G_m],$$
$$X_{0}^{(n)}\simeq [(\MM^{(\infty)}_{1,2}\times S_n)/\GG_{\glue,n}]\simeq [(\MM^{(1)}_{1,2}\times S_n)/\G_m],$$
compatible with the extra $l(\G_m)$-action. As before, this induces isomorphisms of quotients
\begin{equation}\label{torsor-1-sep-isom-eq}
[X_{1,1}^{(n)}/\G_m]\simeq [(\ov{\MM}^{(1)}_{1,1}\times \ov{\MM}_{1,1}^{(1)}\times S_n)/(l(\G_m)\times r(\G_m))],
\end{equation}
\begin{equation}\label{torsor-1-nonsep-isom-eq}
[X_{0}^{(n)}/\G_m]\simeq [(\MM^{(1)}_{1,2}\times S_n)/(l(\G_m)\times r(\G_m))].
\end{equation}

Now we use the following easy observation: if $L$ is a line bundle over $S$ and $\PP\to S$ is the corresponding $\G_m$-torsor,
then the total space of $L$ can be identified with the quotient $(\PP\times \A^1)/\G_m$, where $\la\in \G_m$ acts by $(p,q)\mapsto (\la p, \la q)$,
and the $n$th infinitesimal neighborhood of the zero section in the total space of $L$ can be identified with the subscheme
$$(\PP\times S_n)/\G_m \sub (\PP\times \A^1)/\G_m.$$ 
Similarly, if $L_1,L_2$ are two line bundles, $\PP_1$, $\PP_2$ are the corresponding $\G_m$-torsors,
then 
$$(\PP_1\times_S \PP_2\times S_n)/\G_m^2,$$
where $(\la_1,\la_2)\in \G_m^2$ acts by $(p_1,p_2,q)\mapsto (\la_1 p_1,\la_2 p_2,\la_1\la_2 q)$, can be identified with $n$th infinitesimal neighborhood of the zero section
in the total space of $L_1\ot L_2$.

To deduce the result, we apply this observation to the right-hand sides of \eqref{torsor-1-sep-isom-eq} and \eqref{torsor-1-nonsep-isom-eq}, and use the well known
identification of the normal bundle of the boundary divisor with the product of line bundles associated with the marked points (see \cite[Sec.\ XIII.3]{ACG2}).
\ed

\section{Periods near a separating node}\label{periods-sec}

\subsection{Differentials and cohomology}

In this section we work in the complex analytic category.

Let us consider a family of stable curves $\pi:C\to S$, where $S$ is a disk, the total space $C$ is smooth, and $\pi$ is smooth over $S\setminus 0$.
Let $\om_{C/S}$ denote the relative dualizing sheaf.
It is well known that in this case the higher direct images of the relative logarithmic de Rham complex, which can be identified with 
$$\ov{\HH}^i:=R^i\pi_*[\OO_C\to \om_{C/S}],$$
are vector bundles on $S$ and give a canonical extension of the local systems
$$\HH^i:=R^i\pi_*\pi^{-1}\OO_{S\setminus 0}$$ from $S\setminus 0$ 
(see \cite[Thm.\ 2.18]{Steenbrink}).

Thus, the coboundary map
$$\pi_*(\om_{C/S})\to R^1\pi_*[\OO_C\to \om_{C/S}]$$
extends the standard embedding of $\pi_*\om_{C/S}$ into $R^1\pi_*\pi^{-1}\OO_S$ defining periods over $S\setminus 0$.

Furthermore, we claim that the formation of the higher direct image $\ov{\HH}^1=R^1\pi_*[\OO_C\to \om_{C/S}]$ is compatible with the base change 
from $S$ to the $n$-th infinitesimal neighborhoods $S_n\sub S$ of $0\in S$. Indeed, even though the differential $\OO_C\to \om_{C/S}$ is not $\OO_C$-linear,
computing the direct image using a Cech resolution, one arrives in a standard way to the setup of the cohomology base change formalism as in \cite[Thm.\ II.5, p.46]{Mum-av}.
Thus, the computation of the cohomology of restriction of $[\OO_C\to \om_{C/S}]$ over the special fiber 
(see \cite[Thm.\ 2.18]{Steenbrink}), implies that we have a cohomology base change isomorphism
$$\ov{\HH}^1|_{S_n}\simeq H^1(C_n,[\OO_{C_n}\to \om_{C_n/S_n}]),$$
where $C_n\to S_n$ is the base change family over $S_n$.
It is also well known that the formation of $\pi_*(\om_{C/S})$ for families of stable curves is compatible with the base change.



Now we start with a stable curve $C_0$ which is the union of the smooth components $C_1$ and $C_2$, such that points $p_1\in C_1$ and $p_2\in C_2$ are glued into 
a node $p_0\in C_0$.
We consider any deformation $\pi:\CC\to S_n$ of a stable curve $C_0$ over $S_n$, inducing a nontrivial deformation of the node singularity over $S_1$.
Such a deformation always arises in a fashion described above, so the computation of the map
\begin{equation}\label{Pi-period-map}
\Pi: \pi_*(\om_{\CC/S_n})\to R^1\pi_*[\OO_{\CC}\to \om_{\CC/S_n}]
\end{equation}
would give us information about the period map near $C_0$.

We can compute  $R^i\pi_*[\OO_{\CC}\to \om_{\CC/S_n}]$ using the Cech complex with respect to an acyclic flat covering
$\CC=U_1\cup U_2$, where $U_1=\CC\setminus \{p_0\}$
and $U_2$ is the formal neighborhood of $p_0$ in $\CC$. 
Note that since $C_0\setminus \{p_0\}=(C_1\setminus \{p_1\})\sqcup (C_2\setminus \{p_2\})$ is smooth and affine, $U_1$ is a trivial deformation of $C_0\setminus \{p_0\}$.
On the other hand, we can identify $U_2$ with the standard deformation $\OO_{S_n}[\![x_1,x_2]\!]/(x_1x_2-q)$, where $q$ is a coordinate on $S$ (where $x_i$ is a formal parameter
at $p_i\in C_i$, for $i=1,2$). Then
$\OO(U_{12})=\OO_{S_n}(\!(x_1)\!)\oplus \OO_{S_n}(\!(x_2)\!)$.

\begin{lemma}\label{Cech-degen-lem}
The natural projection induces a quasi-isomorphism
$$\CC[\OO_\CC\to \om_{\CC/S_n}]\to [\OO(U_1)\rTo{d} \om_{\CC/S_n}(U_1)\to \om_{\CC/S_n}(U_{12})/(d\OO(U_{12})+\om_{\CC/S_n}(U_2))].$$
Furthermore,
$$[\OO(U_1)\to \om_{\CC/S_n}(U_1)]$$ is a subcomplex in the above complex, and the embedding induces an isomorphism on $H^1$. Thus, we get an identification 
$$R^1\pi_*[\OO_\CC\to \om_{\CC/S_n}]\simeq \om_{\CC/S_n}(U_1)/d\OO(U_1).$$ 
\end{lemma}

\begin{proof}
First, let us consider the short exact sequence of complexes, where $\CC^\bullet=\CC[\OO_\CC\to \om_{\CC/S_n}]$.
\begin{diagram}
0&\rTo{}&\OO(U_2)&\rTo{}&\OO(U_2)\oplus \om_{\CC/S_n}(U_2)&\rTo{}& \om_{\CC/S_n}(U_2)&\rTo{}&0\\
&&\dTo{}&&\dTo{}&&\dTo{}\\
0&\rTo{}&\CC^0&\rTo{}&\CC^1&\rTo{}&\CC^2&\rTo{}&0\\
&&\dTo{}&&\dTo{}&&\dTo{}\\
0&\rTo{}&\OO(U_1)&\rTo{}&\OO(U_{12})/\OO(U_2)\oplus \om_{\CC/S_n}(U_1)&\rTo{}& \om_{\CC/S_n}(U_{12})/\om_{\CC/S_n}(U_2)&\rTo{}&0\\
\end{diagram}
Since the first row is exact, the first assertion follows from the fact that the map 
$$d: \OO(U_{12})/\OO(U_2)\to \om_{\CC/S_n}(U_{12})/\om_{\CC/S_n}(U_2)$$
is injective. The latter injectivity can be checked as follows: both modules are free over $\OO(S_n)=\C[q]/(q^{n+1})$, so the injectivity follows from
the injectivity modulo $(q)$.

Recall that $U_1$ is a trivial deformation of $(C_1\setminus \{p_1\})\sqcup (C_2\setminus \{p_2\})$.
The map $\om_{\CC/S_n}(U_1)\to \om_{\CC/S_n}(U_{12})$ sends a pair of diffferentials $(\om_1,\om_2)$, where 
$\om_i\in H^0(C_i\setminus\{i\},\om_{C_i})\ot \C[q]/(q^{n+1})$, for $i=1,2$, to their expansions in the formal parameters at $p_1$, $p_2$.
By the residue theorem, the coefficients of $dx_1/x_1$ and $dx_2/x_2$ in these expansions will be trivial.
Hence, the map 
$$\om_{\CC/S_n}(U_1)\to \om_{\CC/S_n}(U_{12})/d\OO(U_{12})$$
 is zero, which implies the second assertion.
\end{proof}

Using Lemma \ref{Cech-degen-lem} and the triviality of the deformation $U_1/S_n$ we can view the map \eqref{Pi-period-map} as a morphism
of $\C[q]/(q^{n+1})$-modules
\begin{equation}\label{Pi-period-map-bis}
\begin{array}{l}
\Pi:H^0(\CC,\om_{\CC/S_n})\to \om_{\CC/S_n}(U_1)/d\OO(U_1)\simeq \\
\bigl(\om_{C_1}(C_1\setminus\{p_1\})/d\OO(C_1\setminus\{p_1\})\oplus \om_{C_2}(C_2\setminus\{p_2\})/d\OO(C_2\setminus\{p_2\})\bigr)\ot \C[q]/(q^{n+1}),
\end{array}
\end{equation}
induced by the restriction map $H^0(\CC,\om_{\CC/S_n})\to \om_{\CC/S_n}(U_1)/d\OO(U_1)$.
Note since the map $\Pi$ in injective modulo $q$, it defines a split embedding of free $\C[q]/(q^{n+1})$-modules.

\subsection{Basis of differentials regular outside a point}\label{basis-diff-sec}

Let $C$ be a smooth projective curve of genus $g\ge 1$, $p\in C$ a point such that $h^1(\OO(gp))=0$ (i.e., $p$ is not a Weierstrass point),
and $z$ a formal parameter at $p$.

\begin{lemma}\label{basis-diff-lem} 
(i) There is a unique basis $\a[0],\ldots,\a[g-1]$ of regular differentials on $C$ with $\a[i]=(z^i+O(z^g))dz$, for $i=0,\ldots,g-1$. 

\noindent
(ii) For each $n\ge 2$, there is a unique rational differential $\om[-n]\in H^0(C,\om_C(np))$ with $\om[-n]=(z^{-n}+O(z^g))dz$.
The differentials $(\a[0],\ldots,\a[g-1],\om[-2],\om[-3],\ldots)$ form a basis of $\om_C(C\setminus p)$.

\noindent
(iii) For every nonzero tangent vector $v$ at $p$ there is a unique formal parameter $z$ at $p$ such that $v(t)=1$ and
$z^{g-1}dz$ extends to a regular differential on $C$, i.e., $\a[g-1]=z^{g-1}dz$.
\end{lemma}

\begin{proof}
(i) The condition $H^1(\OO(gp))=0$ implies that $H^0(\om_C(-gp))=0$ and $h^1(\om_C(-gp))=1$.
Hence, for every $i=0,\ldots,g$, we have $h^0(\om_C(-(g-i)p))=i$. Now we can prove by induction on $i$
that there is a unique basis $\a[g-i],\ldots,\a[g-1]$ of $H^0(\om_C(-(g-i)p))$ with 
$\a[j]=(z^j+O(z^g))dz$. For $i=0$ such a basis is empty. Assume we have such a basis of
$H^0(\om_C(-(g-i)p))$. Let us choose an element 
$\a\in H^0(\om_C(-(g-i+1)p))\setminus H^0(\om_C(-(g-i)p))$. Rescaling $\a$ we can assume
that $\a=(z^{g-i+1}+O(z^{g-i}))dz$. Adding to $\a$ a linear combination of $\a[g-i],\ldots,\a[g-1]$ we 
get the required element $\a[g-i+1]=(z^{g-i+1}+O(z^g))dz$. The uniqueness is clear.

\noindent
(ii) This easily follows from the fact that $h^0(\om_C(np))=g-1+n$ for $n\ge 2$ and
from the fact that elements of $H^0(\om_C(np))$ have zero residue at $p$.

\noindent
(iii) Let us start with an arbitrary formal parameter $z$, and construct a basis of $H^0(\om_C)$ as above. 
Let 
$$\a[g-1]=(z^{g-1}+a_gz^g+\ldots)dz.$$
We have to check that there is a unique series
$u(z)=z+c_2z^2+\ldots$ such that 
$$\a[g-1]=u(z)^{g-1}\cdot du(z)=u(z)^{g-1}u'(z)dz,$$
i.e., 
$$(1+c_2z+c_3z^2+\ldots)^{g-1}\cdot (1+2c_2z+3c_3z^2+\ldots)=1+a_gz+a_{g+1}z^2+\ldots.$$
We can solve this for $(c_i)_{i\ge 2}$ step by step: 
$$(g+1)c_2=a_g, \ \ (g+2)c_3+(g-1)(g-2)c_2^2=a_{g+1}, \text{ etc.}$$
\end{proof}

Let us consider the basis $(\a[0],\ldots,\a[g-1],\om[-2],\om[-3],\ldots)$ of $\om_C(C\setminus p)$.
Let us consider the expansions
$$\a[i]=(z^i+z^g\sum_{n\ge 0}\a_n[i] z^n)dz, \ \text{ for } i=0,\ldots,g-1,$$
$$\om[-i]=(z^{-i}+z^g\sum_{n\ge 0}\om_n[-i] z^n)dz,  \ \text{ for } i\ge 2.$$
Note that the special choice of a parameter $z$ as in Lemma \ref{basis-diff-lem}(iii) corresponds to $\a_n[g-1]=0$ for $n\ge 0$.
From now on we assume that $z$ is chosen in this way.

\begin{lemma}
(i) For each $n\ge 1$, let $f[-g-n]$ denote the unique element of $H^0(C,\OO_C((g+n)p))$, up to additive constant, with $f[g+n]=z^{-g-n}+O(z^{-g})$.
Then up to an additive constant we have the expansion
\begin{align*}
&-f[-g-n]+\frac{1}{z^{g+n}}=\frac{\a_{n-1}[g-2]}{z^{g-1}}+\frac{\a_{n-1}[g-3]}{z^{g-2}}+\ldots+\frac{\a_{n-1}[0]}{z}\\
&+\om_{n-1}[-2]\cdot z+\om_{n-1}[-3]\cdot z^2+\ldots+\om_{n-1}[-g-1]\cdot z^g+O(z^{g+1}).
\end{align*}

\noindent
(ii) For every $n\ge 1$, we have 
\begin{align*}
&df[-g-n]+(g+n)\om[-g-n-1]=\\
&(g-1)\a_{n-1}[g-2]\om[-g]+(g-2)\a_{n-1}[g-3]\om[-g+1]+\ldots+\a_{n-1}[0]\om[-2]\\
&-\om_{n-1}[-2]\a[0]-2\om_{n-1}[-3]\a[1]-\ldots-g\om_{n-1}[-g-1]\a[g-1].
\end{align*}
\end{lemma}

\begin{proof}
(i) This follows from the residue theorem which gives equations
$$\Res_p (f[-g-n]\cdot \a[i])=\Res_p (f[-g-n]\cdot\om[-j])=0.$$

\noindent
(ii) Differentiating the expansion of $f[-g-n]$ found in part (i), we get
\begin{align*}
&\frac{df[-g-n]}{dz}=\frac{g+n}{z^{g+n+1}}+\frac{(g-1)\a_{n-1}[g-2]}{z^g}+\ldots
+\frac{\a_{n-1}[0]}{z^2}\\
&-\om_{n-1}[-2]-2\om_{n-1}[-3]z-\ldots-g\om_{n-1}[-g-1]z^{g-1}+O(z^g),
\end{align*}
which implies that the difference between both sides in the claim identity belongs
to $H^0(C,\om_C(-gp))=0$.
\end{proof}

\begin{cor}\label{H1-forms-cor}
The rational forms $\a[0],\ldots,\a[g-1],\om[-2],\ldots,\om[-g-1]$ form a basis of $H^1(C,\C)$, identified with
the quotient of $\om_C(C\setminus p)/d\OO(C\setminus p)$. For $n\ge 1$, we have equality of classes in $H^1(C,\C)$,
\begin{align*}
& \om[-g-n-1]\equiv \frac{g-1}{g+n}\a_{n-1}[g-2]\om[-g]+\frac{g-2}{g+n}\a_{n-1}[g-3]\om[-g+1]+\ldots+\frac{1}{g+n}\a_{n-1}[0]\om[-2]\\
&-\frac{1}{g+n}\om_{n-1}[-2]\a[0]-\frac{2}{g+n}\om_{n-1}[-3]\a[1]-\ldots-\frac{g}{g+n}\om_{n-1}[-g-1]\a[g-1].
\end{align*}
\end{cor}


\subsection{Differentials near a separating node} 


Let us consider the extended clutching construction for a pair of curves with formal parameters at marked points, $(C_1,p_1,x_1)$, $(C_2,p_2,x_2)$, such that
for $i=1,2$, genus of $C_i$ is $g_i\ge 1$ and $p_i$ is not a Weierstrass point on $C_i$, i.e., $H^1(C_i,\OO(g_ip_i))=0$.
Here we assume that the formal parameters $x_i$ are chosen as in Lemma \ref{basis-diff-lem}(iii). 
Let $\CC$ be the corresponding family of curves over $S_N=\Spec(A_N)$, where $A_N=\C[q]/(q^{N+1})$.
We want to construct a special basis of the bundle $\pi_*\om_{\CC/S_N}$ over $S_N$.

The sections of $\pi_*\om_{\CC/S_N}$ can be described by the quadruples
$(\om_1,\phi_1(x_1),\om_2,\phi_2(x_2))$, 
$\om_i$ being a global form on $(C_i\setminus \{p_i\})_{A_N}$ and $\phi_i(x_i)\in A_N[\![x_i]\!]$, such that
\begin{equation}\label{global-differentials-gluing-eq}
\om_1(x_1)=\phi_1(x_1)dx_1-q \phi_2(\frac{q}{x_1})\frac{dx_1}{x_1^2}, \ \ \om_2(x_2)=\phi_2(x_2)dx_2-q \phi_1(\frac{q}{x_2})\frac{dx_2}{x_2^2}.
\end{equation}
Actually, $\phi_1(x_1)$ and $\phi_2(x_2)$ are uniquely determined by these equations from $\om_1$ and $\om_2$ (as regular parts in their expansions).

We consider the differentials $\om[-n]$ on $C_1$, regular outside $p_1$, defined as in Sec.\ \ref{basis-diff-sec}.
Let us denote by $\La_1\sub \om(C_1\setminus \{p_1\})$ the linear span of $(\om[-n])_{n\ge 2})$. Note that we have a direct sum decomposition
$$\om(C_1\setminus \{p_1\})=H^0(C_1,\om_{C_1})\oplus \La_1.$$
Similarly, we define the subspace $\La_2\sub \om(C_2\setminus \{p_2\})$, where we denote the basis in $\La_2$ by $(\om'[-n])_{n\ge 2}$.

\begin{prop}\label{differential-q-exp-prop} 
(i) For every regular differential $\a$ on $C_1$ there exists a unique global section $\Phi_1(\a)$ of $\om_{\CC/S_N}$ with
$$\om_1\equiv \a \mod qA_N\ot \La_1, \ \ \om_2\equiv 0 \mod qA_N\ot \La_2.$$
This gives a linear embedding $\Phi_1:H^0(C_1,\om_{C_1})\to H^0(\om_{\CC/S_N})$.

\noindent
(ii) For $i=0,\ldots,g_1-1$, one has 
$\Phi_1(\a[i])=(\om_1,\om_2)$ with
$$\om_1\equiv \a[i]+q^{i+g_2+2}\om'_0[-i-2]\cdot\om[-g_2-2]+\ldots+q^{g_1+g_2+1}\om'_{g_1-i-1}[-i-2]\cdot \om[i-g_1-g_2-1] \mod q^{g_1+g_2+2},$$
$$\om_2\equiv -q^{i+1}\om'[-i-2]-q^{g_1+1}\a_0[i]\cdot\om'[-g_1-2]-\ldots-q^{g_1+g_2+1}\a_{g_2}[i]\cdot \om'[-g_1-g_2-2].$$

\noindent
(iii) Similarly, we define a linear embedding
$\Phi_2:H^0(C_2,\om_{C_2})\to H^0(\om_{\CC/S_N})$. The induced map 
$$(\Phi_1,\Phi_2):H^0(C_1,\om_{C_1})\ot A_N\oplus H^0(C_2,\om_{C_2})\ot A_N\to H^0(\CC,\om_{\CC/S_N})=H^0(S_N,\pi_*\om_{\CC/S_N})$$
is an isomorphism.
\end{prop}

\begin{proof}
(i) We solve the equations \eqref{global-differentials-gluing-eq} order by order in $q$. 
Modulo $q$ we require to have $\om_1\equiv \a$, $\om_2\equiv 0$, so $\phi_1(x_1)dx_1$ is the expansion of $\a$ at $p_1$ and $\phi_2=0$.
Assume now $N\ge 1$ and we already know $(\om_1,\phi_1,\om_2,\phi_2)$ modulo $q^N$.
Since $\phi_i(x_i)$ are regular, knowing $\phi_i$ modulo $q^N$ determines uniquely polar parts of $\om_1(x_1)$ and $\om_2(x_2)$ modulo $q^{N+1}$.
More precisely, we have
$$\om_1(x_1)_{\le -1}\equiv -q\phi_2(\frac{q}{x_1})\frac{dx_1}{x_1^2} \mod (q^{N+1}),$$
and similarly for the polar part of $\om_2(x_2)$.
Since $\phi_2(\frac{q}{x_1})\frac{dx_1}{x_1^2}\in \sspan(x_1^{-2},x_1^{-3},\ldots)$,
there exists a unique $\la_1\in qA_N\ot \La$ with
$$(\la_1)_{\le -1}=-q\phi_2(\frac{q}{x_1})\frac{dx_1}{x_1^2}.$$
Hence, the unique solution for $\om_1$ is $\om_1=\a+\la_1$. The same argument gives a unique solution for $\om_2$.
Now $\phi_i(x_i)$ modulo $q^{N+1}$ is determined as the regular part of the expansion of $\om_i(x_i)$:
$$\phi_i(x_i)dx_i=\om_1(x_1)_{\ge 0}.$$

\noindent
(ii) We just have to check that for these $\om_1$, $\om_2$ and for
$\phi_i(x_i)dx_i=\om_i(x_i)_{\ge 0}$, $i=1,2$, equations \eqref{global-differentials-gluing-eq} are satisfied modulo $q^{g_1+g_2+2}$.

First, let us calculate $\phi_1(\frac{q}{x_2})$ modulo $q^{g_1+g_2+1}$.
Recall that $\om[-g_2-2]_{\ge 0}=O(x_1^{g_1})$, hence only the term $\a[i]$ in the formula for $\om_1$ will contribute:
$$\phi_1(\frac{q}{x_2})\equiv \frac{q^i}{x_2^i}+\sum_{n\ge 0}\a_n[i]\frac{q^{n+g_1}}{x_2^{n+g_1}} \mod q^{g_1+g_2+1}.$$
Hence, we see that $-q\phi_1(\frac{q}{x_2})dx_2/x_2^2$ matches the polar part of $\om_2$ modulo $q^{g_1+g_2+2}$.

Similarly, when we compute $\phi_2(\frac{q}{x_1})$ modulo $q^{g_1+g_2+1}$, only the first term in the formula for $\om_2$
will contribute:
$$\phi_2(\frac{q}{x_1})\equiv -\sum_{n\ge 0}\om'_n[-i-2]\frac{q^{n+g_2+i+1}}{x_1^{n+g_2}} \mod q^{g_1+g_2+1}.$$
Again this implies that $-q\phi_2(\frac{q}{x_1})dx_1/x_1^2$ matches the polar part of $\om_1$ modulo $q^{g_1+g_2+2}$.

\noindent
(iii) Note that $H^0(S_N,\pi_*\om_{\CC/S_N})$ is a free $A_N$-module of a finite rank,
and its reduction modulo $A_N/qA_N$ gets identified with 
$$H^0(C_0,\om_{C_0})\simeq H^0(C_1,\om_{C_1})\oplus H^0(C_2,\om_{C_2}).$$ 
Now the assertion follows from the fact that our map $(\Phi_1,\Phi_2)$ reduces to the identity modulo $q$.
\end{proof}

Let 
$$\Pi:[H^0(C_1,\om_{C_1})\oplus H^0(C_2,\om_{C_2})]\ot A_N\to 
H^1(\CC,\pi^{-1}\OO_{S_N})\simeq [H^1(C_1,\C)\oplus H^1(C_2,\C)]\ot A_N$$
denote the composition of the isomorphism $(\Phi_1,\Phi_2)$ with the map \eqref{Pi-period-map-bis}.
We can expand $\Pi$ in powers of $q$: 
$$\Pi=\Pi_0+q\Pi_1+q^2\Pi_2+\ldots$$
Note that $\Pi_0$ is just the sum of the standard embeddings $H^0(C_j,\om_{C_j})\to H^1(C_j,\C)$, over $j=1,2$.
From Proposition \ref{differential-q-exp-prop} we can calculate $\Pi_j$ for $j\le g_1+g_2+1$.

\begin{cor}\label{periods-small-order-cor}
For $1\le j\le g_1+g_2+1$, and $0\le i\le g_1-1$, one has
$$\Pi_j(\a[i])=
(\om'_{j-i-g_2-2}[-i-2]\om[-j+i],-\de_{j,i+1}\om'[-i-2]-\a'_{j-g_1-1}[i]\cdot \om'[-j-1]),$$
where the term with the negative index is considered to be zero,
and the image of $\om[-n]$ in $H^1(C_1,\C)$ for $n\ge g_1+1$ is given by Corollary \ref{H1-forms-cor} 
(the image of $\om'[-n]$ in $H^1(C_2,\C)$ is given by similar formulas). 

In particular, for $1\le j\le \min(g_1,g_2)$, one has
$$\Pi_j(\a[j-1],0)=(0,-\om'[-j-1]), \ \ \Pi_j(0,\a'[j-1])=(-\om[-j-1],0)$$
and $\Pi_j$ is zero on all other basis vectors.
\end{cor}

Looking only at $\Pi_1$ we get the following description of the tangent map to the period mapping, viewed as a map from the moduli space to
the Grassmannian.
Note that if $C_0$ is the stable curve glued from $C_1$ and $C_2$ as above (for $q=0$) then the period mapping sends $C_0$ to the point of the Grassmannian 
corresponding to the subspace
$$H^0(C_1,\om_{C_1})\oplus H^0(C_2,\om_{C_2})\sub H^1(C_1,\C)\oplus H^1(C_2,\C).$$
The tangent space to the Grassmannian at this point is identified with
$$T_{C_0}G=\Hom_{\C}(H^0(C_1,\om_{C_1})\oplus H^0(C_2,\om_{C_2}), H^1(C_1,\OO_{C_1})\oplus H^1(C_2,\OO_{C_2}).$$

\begin{prop}
Let $(C_1,p_1,x_1)$ (resp., $(C_2,p_2,x_2)$) be a smooth curve of genus $g_1\ge 1$ (resp., $g_2\ge 1$) with a marked point and a formal parameter at it.
Assume that $p_1$ (resp., $p_2$) is not a Weierstrass point. Consider the functionals 
$$\phi_i:H^0(C_i,\om_{C_i})\to \om_{C_i}|_{p_i}\simeq \C, i=1,2$$
and the vectors
$$\de_i:\C\simeq H^0(C_i,\OO_{C_i}(p_i)/\OO_{C_i})\to H^1(C_i,\OO_{C_i}).$$
Here we use the trivialization of $\om_{C_i}|_{p_i}$ (resp., of $\OO_{C_i}(p_i)/\OO_{C_i}$) induced by $dx_i$ (resp., $1/x_i$).
Then the derivative of the period map at $C_0=C_1\cup C_2$ with respect to the parameter $q$ (coming from the extended clutching construction)
is given by
$$(\de_2\phi_1,\de_1\phi_2)\in \Hom_{\C}(H^0(C_1,\om_{C_1}),H^1(C_2,\OO_{C_2}))\oplus \Hom(H^0(C_2,\om_{C_2}),H^1(C_1,\OO_{C_1}))\sub T_{C_0}G.$$
\end{prop}

\begin{proof} 
By Corollary \ref{periods-small-order-cor}, we have
$$\Pi_1(\a[0],0)=(0,-\om'[-2]), \ \ \Pi_1(0,\a'[0])=(-\om[-2],0)$$
and $\Pi_1$ is zero on all other basis vectors.

It is clear from the definition that $\phi_1$ is the functional on $H^0(C_i,\OO_{C_i}(p_i)$ sending $\a[i]$ with $i>0$ to zero, and sending $\a[0]$ to $1$.
It remains to check $\de_1$ is the image of $-\om[-2]$ under the composition
$$\om_{C_1}(C\setminus \{p_1\})\to H^1(C_1,[\OO_{C_1}\to \om_{C_1}])\to H^1(C_1,\OO_{C_1}).$$ 
This map sends a differential $\eta$ to the class of the Laurent series $\phi(x_1)$ such that $\eta=d\phi(x_1)$ in the formal neighborhood of $p_1$.
Since $-\om[-2]=-\frac{1}{x_1^2}+\ldots$, we have in our case $\phi(x_1)=-\frac{1}{x_1}+O(1)$, which implies our statement.
\end{proof}





\subsection{Separating node boundary in $g=2$ moduli}

Now we specialize to the case $g_1=g_2=1$.

First of all, for $g=1$ we can express all objects in Sec.\ \ref{basis-diff-sec} in terms of elliptic functions.
Let $C=\C/(\Z+\Z\tau)$ be an elliptic curve with $0$ as the marked point.
The canonical formal parameter of Lemma \ref{basis-diff-lem}(iii) is just the coordinate $z$ coming from the complex plane, with
$\a[0]=dz$.

The rational functions $f[-n]$, $n\ge 2$ (with a particular normalization of additive constants), are expressed in terms of 
the rescaled derivatives of the Weierstrass $\wp$-function:
$$f[-n]=(-1)^n \frac{1}{(n-1)!}\wp^{(n-2)}(z)-\frac{c_{n-2}}{n-1}=\frac{1}{z^n}+
(-1)^n\sum_{m\ge 1}{m+n-2\choose m}\frac{c_{m+n-2}}{n-1}z^m,$$
where 
$$\wp(z)=\frac{1}{z^2}+c_2z^2+c_4z^4+\ldots$$
(where $c_i=c_i(\tau)$ and $c_i=0$ for odd $i$ and $c_0=0$).
The above normalization of $f[-n]$ satisfies $f[-n]=z^{-n}+O(z)$.
Hence, the rational differentials $\om[-n]$ are given by
$$\om[-n]=f[-n]dz.$$

The basis of $H^1(C,\C)\simeq \om_C(C\setminus p)/d\OO(C\setminus p)$ is given by 
$\a[0]=dz$ and $\om[-2]=\wp(z)dz$.
Note that for $n\ge 3$ we have
$$\om[-n]\equiv -\frac{c_{n-2}}{n-1}\cdot dz$$
in $H^1(C,\C)$.

Now we consider the extended clutching applied to the elliptic curves $C_1=\C/(\Z+\Z\tau_1)$ and $C_2=\C/(\Z+\Z\tau_2)$ with the origins as the marked points.

We claim that $\Phi_1(dx_1)$ can be found in the form
$$\Phi_1(dz)=([1+\sum_{n\ge 1} a_{2n}(q)f[-2n](x_1,\tau_1)]dx_1, -q[f[-2](x_2,\tau_2)+\sum_{n\ge 2}b_{2n}(q)f[-2n](x_2,\tau_2)]dx_2),$$
for some series $a_{2n}(q)$, $n\ge 1$ and $b_{2n}(q)$, $n\ge 2$, depending on $\tau_1$ and $\tau_2$.
Indeed, we have
$$\phi_1(x_1)=1+\sum_{n\ge 1} a_{2n}(q)f[-2n](x_1,\tau_1)_{>0}, \ \ \phi_2(x_2)=-qf[-2](x_2,\tau_2)_{>0}-q\sum_{n\ge 2}b_{2n}(q)f[-2n](x_2,\tau_2)_{>0},$$
and equations \eqref{global-differentials-gluing-eq} can be rewritten as
$$\sum_{n\ge 1} \frac{a_{2n}(q)}{x_1^{2n}}=-\frac{q}{x_1^2}\phi_2(\frac{q}{x_1})=\frac{q^2}{x_1^2}\bigl(f[-2](\frac{q}{x_1},\tau_2)_{>0}+
\sum_{n\ge 2}b_{2n}(q)f[-2n](\frac{q}{x_1},\tau_2)_{>0}\bigr),$$
$$\sum_{n\ge 2}\frac{b_{2n}(q)}{x_2^{2n}}=\frac{1}{x_2^2}\sum_{n\ge 1} a_{2n}(q)f[-2n](\frac{q}{x_2},\tau_1)_{>0}.$$
These are equivalent to the following recursion for $(a_{2n})_{n\ge 1}$, $(b_{2n})_{n\ge 2}$:
$$a_{2n}=q^{2n}[c_{2n-2}(\tau_2)+\sum_{m\ge 2}b_{2m}\cdot {2m+2n-4\choose 2n-2}\frac{c_{2m+2n-4}(\tau_2)}{2m-1}],$$
$$b_{2n}=q^{2n-2}\sum_{m\ge 2}a_{2m}\cdot {2m+2n-4\choose 2n-2}\frac{c_{2m+2n-4}(\tau_1)}{2m-1}, \text{ for } n\ge 2,$$
with the initial condition $a_2=0$.
This recursion has a unique solution, which proves our claim.

For example, we have
$$a_4\equiv c_2(\tau_2)q^4+4c_4(\tau_1)c_2(\tau_2)c_4(\tau_2)q^{10} \mod q^{12}, \ a_6\equiv c_4(\tau_2)q^6 \mod q^{12}, \ 
a_8\equiv c_6(\tau_2)q^8 \mod q^{12},$$
$$b_4\equiv 2c_4(\tau_1)c_2(\tau_2)q^6+ 3c_6(\tau_1)c_4(\tau_2)q^8 \mod q^{10}, \
b_6\equiv 5c_6(\tau_1)c_2(\tau_2)q^8 \mod q^{10},$$
while all higher $a_{2n}$ (resp., $b_{2n}$) are zero modulo $q^{12}$ (resp., $q^{10}$).

The map $\Pi$ sends the differential $(dx_1,0)$ to the class
$$([1-\sum_{n\ge 2}a_{2n}(q)\frac{c_{2n-2}(\tau_1)}{2n-1}]dx_1,-q\om[-2]+q[\sum_{n\ge 2}b_{2n}(q)\frac{c_{2n-2}(\tau_2)}{2n-1}]dx_2)$$
in $(H^1(C_1,\C)\oplus H^1(C_2,\C))\ot \C[q]/(q^{N+1})$ (the formula for $\Pi(0,dx_2)$ is obtained by switching the roles of $C_1$ and $C_2$).

Thus, modulo $q^{11}$ we have
\begin{align*}
&\Pi(dx_1,0)\equiv \\
&([1-\frac{c_2(\tau_1)c_2(\tau_2)}{3}q^4-\frac{c_4(\tau_1)c_4(\tau_2)}{5}q^6-\frac{c_6(\tau_1)c_6(\tau_2)}{7}q^8-
\frac{4}{3}c_2(\tau_1)c_4(\tau_1)c_2(\tau_2)c_4(\tau_2)q^{10}]dx_1, \\
&-q\om[-2]+[\frac{2}{3}c_4(\tau_1)c_2(\tau_2)^2q^7+2c_6(\tau_1)c_2(\tau_2)c_4(\tau_2)q^9]dx_2).
\end{align*}

\end{document}